\documentclass[smallextended,envcountsect]{svjour3}
\usepackage{pifont}
\usepackage{epsfig,color,graphicx,subfigure,amsmath,amssymb,multirow,lscape}
\usepackage[colorlinks,linkcolor=blue,citecolor=blue]{hyperref}
\usepackage{tikz,cite}
\usetikzlibrary{arrows, automata}
\smartqed
\usepackage{graphicx}
\usepackage{amsmath,amssymb}
\journalname{}

\def\pn{\par\smallskip\noindent}

\newtheorem{thm}{Theorem}[section]
\newtheorem{lem}{Lemma}[section]

\newtheorem{rem}{Remark}[section]
\newtheorem{fact}{Fact}[section]
\newtheorem{algo}{Algorithm}[section]
\newtheorem{prob}{Problem}[section]
\newtheorem{asmp}{Assumption}[section]

\def\be{\begin{eqnarray}}
\def\ee{\end{eqnarray}}
\def\ben{\begin{eqnarray*}}
\def\een{\end{eqnarray*}}
\def\ba{\begin{array}}
\def\ea{\end{array}}
\def\bi{\begin{itemize}}
\def\ei{\end{itemize}}
\def\proof {\pn {Proof.} }

\def\endproof{\hfill $\Box$ \vskip .5cm}

\def\cL{{\mathcal L}}

\def\cN{{\mathcal N}}
\def\cO{{\mathcal O}}

\def\cS{{\mathcal S}}

\def\bR{{\mathbb R}}
\def\bN{{\mathbb N}}
\def\bC{{\mathbb C}}

\def\prox{{\rm Prox}}

\def\[{\begin{equation}}
\def\]{\end{equation}}

\newcommand{\D}{\Delta}

\newcommand{\R}{\mathbb R}

\newcommand{\lr}[1]{\left\langle #1\right\rangle}

\begin{document}

\title{A golden ratio primal-dual algorithm for structured convex optimization}
\titlerunning{A golden ratio primal dual algorithm}
\author{Xiaokai Chang$^{1}$
\and  Junfeng Yang$^2$
}

\institute{%\ding {41} Xiaokai Chang \at xkchang@lut.cn
          % \and
         % Junfeng Yang \at jfyang@nju.edu.cn
1 \ \   School of Science, Lanzhou University of Technology, Lanzhou, Gansu, P. R. China. Email: xkchang@lut.cn. \\
2  \ \  Department of Mathematics, Nanjing University, Nanjing, P. R. China. This author was supported by the NSFC grants 11771208 and 11922111. Email: jfyang@nju.edu.cn. }

\date{Received: date / Accepted: date}

\maketitle

\begin{abstract}
  We design, analyze and test a golden ratio primal-dual algorithm (GRPDA) for solving structured convex optimization problem, where the objective function is the sum of two closed proper convex functions, one of which involves a composition with a linear transform. GRPDA preserves all the favorable features of the classical primal-dual algorithm (PDA), i.e., the primal and the dual variables are updated in a Gauss-Seidel manner, and the per iteration cost is dominated by the evaluation of the proximal point mappings of the two component functions and two matrix-vector multiplications. Compared with the classical PDA, which takes an extrapolation step, the novelty of GRPDA is that it is constructed based on a convex combination of essentially the whole iteration trajectory. We show that GRPDA converges within a broader range of parameters than the classical PDA, provided that the reciprocal of the convex combination parameter is bounded above by the golden ratio, which explains the name of the algorithm. An $\cO(1/N)$ ergodic convergence rate result is also established based on the primal-dual gap function, where $N$ denotes the number of iterations. When either the primal or the dual problem is strongly convex, an accelerated GRPDA is constructed to improve the ergodic convergence rate from $\cO(1/N)$ to $\cO(1/N^2)$. Moreover, we show for regularized least-squares and linear equality constrained problems that the reciprocal of the convex combination parameter can be extended from the golden ratio to $2$ and meanwhile a relaxation step can be taken. Our preliminary numerical results on LASSO, nonnegative least-squares and minimax matrix game problems, with comparisons to some state-of-the-art relative algorithms, demonstrate the efficiency of the proposed algorithms.
\end{abstract}

\keywords{Structured convex optimization \and saddle point problem \and primal-dual algorithm \and golden ratio \and acceleration \and convergence rate \and fixed point iteration}

\subclass{49M29 \and  65K10 \and 65Y20 \and 90C25 }

\section{Introduction}
\label{sec_introduction}
Let $\R^p$ and $\R^q$ be finite-dimensional Euclidean spaces, each endowed with an inner product and the induced norm denoted by $\langle\cdot, \cdot\rangle$ and $\|\cdot\| =\sqrt{\langle\cdot,\cdot\rangle}$, respectively.
In this paper, we consider structured convex optimization problem of the form
\be\label{primal}
\min_{x\in \R^q} f(Kx)+g(x),
\ee
where  $f: \R^p\rightarrow(-\infty, +\infty]$ and $g: \R^q \rightarrow(-\infty, +\infty]$ are extended real-valued closed proper convex functions \cite{Rockafellar1970Convex}, and $K\in \R^{p\times q}$ is a linear operator from $\R^q$ to $\R^p$.
Problem \eqref{primal} arises from numerous applications, including signal and image processing, machine learning, statistics, mechanics and economics, to name a few, see, e.g., \cite{Chambolle2011A,Bouwmans2016Handbook,Yang2011Alternating,Hayden2013A,Bertsekas1982Projection}
and the references therein.
%and given by
%\begin{equation*}
%  f^*(y) = \sup_{u\in Y} \langle y, u\rangle - f(u), \quad y \in Y.
%\end{equation*}
The saddle point or primal-dual form of \eqref{primal} reads
\be\label{pd-prob}
\min_{x\in \R^q}\max_{y\in \R^p} ~~\mathcal{L}(x,y):=g(x)+\langle Kx,y\rangle-f^*(y),
\ee
where $f^*(y) = \sup_{u\in \R^p} \langle y, u\rangle - f(u)$, $y\in \R^p$, is the Legendre-Fenchel conjugate of $f$.
Saddle point problems are ubiquitous in optimization as they provide a very convenient way to represent many nonsmooth problems. In particular, \eqref{pd-prob} is intrinsically related to the Fenchel dual problem of \eqref{primal}, which is given by
\be\label{dual}
\max_{y\in \R^p} -f^*(y) -g^*(-K^\top y),
\ee
where $K^\top$ denotes the matrix transpose or adjoint operator of $K$.
%
% and it in turn often allows to improve the complexity rates from $\cO(1/\sqrt{N})$ to $\cO(1/N)$.
%

We note that in this paper we restrict the decision variable $x$ in $\R^q$, yet all the analysis can be extended to any finite dimensional real Euclidean spaces, e.g., the matrix space $\R^{m\times n}$ with trace inner product and the induced Frobenius norm.
In the rest of this section, we define some notation, make our assumptions, review some algorithms for solving \eqref{primal}-\eqref{dual} that are closely related to this work, and summarize our contributions and the organization of this paper.

\subsection{Notation and assumptions}
%We first define some notation.
%
% Throughout this paper, we let $\|\cdot\|$ be the Euclidean norm induced by the inner product of the underlying space.
%
The matrix and vector transpose operation is denoted by superscript ``$\top$".
The operator norm of $K$ is denoted by $L$, i.e., $L := \|K\| = \|K^\top\| = \sup \{\langle Kx, y\rangle: \|x\|=\|y\|=1,\, x\in \R^q,\, y\in \R^p\}$. Throughout this paper, we denote the golden ratio by $\phi$, i.e., $\phi = {\sqrt{5}+1 \over 2}$.
Let $h$ be any extended real-valued closed proper convex function defined on a finite dimensional Euclidean space $\R^m$.
The effective domain of $h$ is denoted by $\text{dom}(h) := \{ x\in\R^m: h(x) < +\infty\}$, and the subdifferential of $h$ at $x\in \R^m$ is denoted by $\partial h(x) := \{\xi\in\R^m: \, h(y) \geq h(x) + \langle \xi, y-x\rangle \text{~for all~}y\in\R^m\}$. Furthermore, for $\lambda >0$, the proximal point mapping of $\lambda h$ is given by
\begin{eqnarray*}\label{def:prox}
  \prox_{\lambda h}(x) := \arg\min_{y\in \R^m } \Big\{h(y) + {1\over 2\lambda }\|y-x\|^2\Big\}, \quad x\in \R^m.
\end{eqnarray*}
Since $h$ is closed proper and convex, for any $\lambda >0$, $\prox_{\lambda h}$ is uniquely well defined everywhere.
The indicator function of a set $C$ is denoted by $\iota_C(x)$, i.e., $\iota_C(x) = 0$ if $x\in C$ and $+\infty$ if otherwise.
The relative interior of $C$ is denoted by $\text{ri}(C)$.
The identity operator or identity matrix is denoted by $I$, whose domain or order is clear from the context.
The zero vector or matrix is simply denoted by $0$. The composition of two operators is denoted by ``$\circ$".
The stagnation of two column vectors $u$ and $v$ is also denoted by $(u; v)$, i.e., $(u; v) = (u^\top, v^\top)^\top$.
The sequence of positive natural numbers is denoted by $\bN = \{1,2,3,\ldots\}$.
Other notation will be specified later.

Throughout the paper, we make the following blanket assumptions.
\begin{asmp}\label{asmp-1}
Assume that the set of solutions of \eqref{primal} is nonempty and,
in addition, there exists $\tilde x\in\text{ri}(\text{dom}(g))$ such that $K\tilde x\in\text{ri}(\text{dom}(f))$.
\end{asmp}

Under Assumption \ref{asmp-1}, it follows from \cite[Corollaries 28.2.2 and 28.3.1]{Rockafellar1970Convex} that
${\bar x}\in \R^q$ is a solution of \eqref{primal} if and only if there exists ${\bar y}\in \R^p$ such that $({\bar x},{\bar y})$ is a saddle point of $\cL(x,y)$, i.e.,
$\cL({\bar x},y)\leq\cL({\bar x},{\bar y})\leq\cL(x,{\bar y})$ for all $(x,y)\in \R^q\times \R^p$, and furthermore, such $\bar y$ is an optimal solution of the dual problem \eqref{dual}.
Throughout this paper, we denote the set of solutions of \eqref{pd-prob} by $\cS$, which is nonempty under  Assumption \ref{asmp-1} and
characterized by
\[\label{def:cS}
\cS := \{({\bar x},{\bar y})\in \R^q\times \R^p: \; 0 \in \partial g({\bar x}) + K^\top {\bar y} \text{~~and~~} 0\in \partial f^*({\bar y}) - K{\bar x}\}.
\]

In many applications including signal and image processing and machine learning, the two component functions in \eqref{primal} enforce, respectively, data fitting and regularization. In such cases, $f$ and $g$ usually preserve simple structures so that their proximal point mappings can be evaluated efficiently. Examples of such functions are abundant, see, e.g., \cite[Chapter 6]{Beck2017book}. We therefore make the following assumption.
\begin{asmp}\label{asmp-2}
  Assume that the proximal point mappings of the component functions $f$ and $g$ either have closed form formulas or can be evaluated efficiently.
\end{asmp}

\subsection{Related algorithms}\label{sc:related-algorithms}
To solve \eqref{primal}-\eqref{dual} simultaneously, one may resort to the well known alternating direction method of multipliers (ADMM) \cite{GM75,GabM76,Jonathan1992On,Lions1979Splitting},
%
%\comm{Here I added the two original work of ADMM, i.e., \cite{GM75} and \cite{GabM76}. As you can see, they do not appear at the beginning of the bracket even though I cited them first. In optimization community, people know that Boyd's paper has little contribution to ADMM, even though it has very high citation number due to it's a good review paper. Since Boyd's paper \cite{Boyd2010Distributed} appear first in this bracket, this may give the referee an impression that we are not professional enough.}
%
the primal-dual algorithm (PDA) \cite{Chambolle2011A,He2012Convergence,Pock2011Diagonal,Esser2010General} and their accelerated and generalized variants \cite{Liu2018Acceleration,Malitsky2018A}.
Since the literature on numerical algorithms for solving \eqref{primal}-\eqref{dual} has become so vast, a thorough overview is far beyond the focus of this paper. Instead, we next review only some primal-dual type algorithms that are most closely related to our work.

A main feature of primal-dual type algorithms is that both the primal and the dual variables are updated at each iteration, and thus the primal and the dual problems are solved simultaneously. Among others, ADMM \cite{GM75,GabM76} has been well studied in the literature \cite{Lions1979Splitting,Jonathan1992On,Fazel2013Hankel} and widely used in practice \cite{Boyd2010Distributed}. The main difficulty encountered by ADMM when applied to (a reformulation of) \eqref{primal} is that a subproblem of the form $\min_{x\in \R^q}  \frac{1}{2}\|Kx-b_n\|^2 + g(x)$ needs to be solved at each iteration for some $b_n\in \R^p$ varying with the iteration counter $n$. We note that even though $\prox_g$ is easy to evaluate, this problem, which involves a linear operator $K$, has to be solved iteratively in general. On the other hand, when $f(Kx)$ has a least-squares structure, i.e., $f(Kx) = {1\over 2}\|Kx-b\|^2$ for some $b\in \R^p$, ADMM can be applied to the equivalent problem $\min_{x,y}\{ f(Kx) + g(y): \; x = y\}$. In this case, the $x$-subproblem appears as a least-squares problem and is equivalent to solving a linear system of equations with coefficient matrix of the form $I + \rho K^\top K$ for some $\rho>0$, which could be expensive or even prohibitive for large scale problems.

The simplest primal-dual type algorithm for solving \eqref{primal}-\eqref{dual}, which does not require to solve any subproblem iteratively or any linear system of equations, is probably the classical Arrow-Hurwicz method \cite{Uzawa58}, which, started at $(x_0,y_0)\in \R^q \times \R^p$, iterates as
\be\label{pda_basic0}
\left\{
\ba{l}
x_{n}=\prox_{\tau g}(x_{n-1}-\tau K^\top y_{n-1}), \smallskip\\
y_{n}=\prox_{\sigma f^*}(y_{n-1}+\sigma K x_{n}),
\ea\right.
\ee
for $n\geq 1$. Here $\tau>0$ and $\sigma>0$ are step size parameters. As a conjugate function, $f^*$ is always closed and convex. Furthermore,
$f^*$ is also proper since $f$ is proper and convex. Thus, $\prox_{\sigma f^*}$ is uniquely well defined everywhere. The iterative scheme \eqref{pda_basic0} is also known as primal dual hybrid gradient method in image processing community, see \cite{ZhC08cam,Esser2010General,Chambolle2011A}. The main computational cost per iteration is the evaluation of two proximal point mappings and two matrix-vector multiplications. %This is an very attractive feature since it does not need to solve any subproblems iteratively or any linear system of equations.
The convergence of the Arrow-Hurwicz method was studied in \cite{Esser2010General} with very small step sizes, and $\cO(1/\sqrt{N})$ rate of convergence, measured by primal-dual gap, was obtained in \cite{Chambolle2011A,Nedic2009Subgradient} when the domain of $f^*$ is assumed to be bounded. However, the Arrow-Hurwicz method does not converge in general, see \cite{He2012Convergence} for a divergent example.

Chambolle and Pock \cite{Chambolle2011A,Pock2011Diagonal} adopted an extrapolation step after obtaining $x_n$ in \eqref{pda_basic0} to obtain ${\bar x}_{n}=x_{n}+\delta (x_{n}-x_{n-1})$ for some $\delta\in[0,1]$, which is then used to replace $x_n$ in \eqref{pda_basic0} in the update of $y_n$. The resulting scheme is nowadays widely accepted as PDA and appears as
\be\label{pda_basic}
\left\{
\ba{l}
x_{n}=\prox_{\tau g}(x_{n-1}-\tau K^\top y_{n-1}), \smallskip\\
{\bar x}_{n}=x_{n}+\delta (x_{n}-x_{n-1}), \smallskip\\
y_{n}=\prox_{\sigma f^*}(y_{n-1}+\sigma K {\bar x}_{n}).
\ea\right.
\ee
In the case $\delta=1$, the convergence of (\ref{pda_basic}) was established in \cite{Chambolle2011A} under the condition $\tau\sigma L^2 < 1$, where $L = \|K\|$.
Note that, for $\delta=1$ the scheme \eqref{pda_basic} is referred to as a split inexact Uzawa method by
Esser, Zhang and Chan \cite{Esser2010General}, where the connection of PDA with preconditioned or linearized ADMM has been revealed, see also \cite{Chambolle2011A,Shefi2014Rate}. Later, it was shown in \cite{He2012Convergence} that PDA \eqref{pda_basic} can be viewed as a weighted proximal point method applied to the variational inequality (VI) representation of the optimality conditions of \eqref{pd-prob}.
The overrelaxed, inertial and accelerated versions of (\ref{pda_basic}) were investigated in \cite{Chambolle2016ergodic}. See also \cite{Chambolle2018STOCHASTIC} for a stochastic variant of PDA.

The effectiveness of the PDA scheme \eqref{pda_basic} is essentially guaranteed by the extrapolation or inertial step ${\bar x}_{n}=x_{n}+\delta (x_{n}-x_{n-1})$ since the Arrow-Hurwicz method \eqref{pda_basic0}, which corresponds to $\delta=0$, fails to converge in general. For a fixed $\delta\in[0,1]$, the scheme \eqref{pda_basic} can be written abstractly as $(x_n,y_n) = T(x_{n-1},y_{n-1})$ for some suitably defined mapping $T$.
Therefore, PDA can be viewed as a one-step fixed point iterative method.
Recently, Malitsky \cite{Malitsky2019Golden} proposed a fully adaptive forward-backward type splitting algorithm, called golden ratio algorithm, for solving mixed VI problem, where a novel convex combination technique is introduced. It is shown that the golden ratio algorithm preserves $\cO(1/N)$ ergodic convergence and R-linear convergence under an error bound condition.
The mixed VI problem is to find $z^*\in \R^m $ such that
\begin{equation}\label{mVI}
\theta(z) - \theta(z^*) + \langle z-z^*, F(z^*)\rangle \geq 0, \quad \forall z \in \R^m ,
\end{equation}
where $\theta: \R^m  \rightarrow (-\infty, +\infty]$ is a closed proper convex function, and $F: \R^m  \rightarrow \R^m $ is a monotone mapping.
Given an initial point $z_0\in \R^m $ and let ${\bar z}_{0} = z_0$, a basic version of the golden ratio algorithm \cite{Malitsky2019Golden} applied to \eqref{mVI} iterates for $n\geq 1$ as
\begin{equation}\label{gr-alg}
\left\{
\begin{array}{rcl}
  \bar{z}_n &=&  {\phi - 1 \over \phi} z_{n-1} + {1\over \phi} {\bar z}_{n-1}, \smallskip \\
  z_{n} &=& \prox_{\tau \theta}(\bar{z}_n - \tau F(z_{n-1})),
\end{array}
\right.
\end{equation}
where $\tau >0$ is a step size parameter and $\phi $ is the golden ratio.
By induction, $\bar{z}_n$ is a convex combination of $\{z_i: i=0,1,\ldots,n-1\}$.
As a result, the update from $z_{n-1}$ to $z_n$ in \eqref{gr-alg} is essentially dependent on the whole iteration trajectory, which is vastly different from the one-step iterative scheme \eqref{pda_basic}.
Note that the primal-dual problem \eqref{pd-prob} is equivalent to  \eqref{mVI} with $\R^m  = \R^q\times \R^p$, endowed with the natural inner product $\langle (x,y), (u,v)\rangle := \langle x, u\rangle + \langle y, v\rangle$ for $(x,y), (u,v)\in \R^q\times \R^p$,
\begin{equation}\label{mVI-pd}
\theta(z) := g(x) + f^*(y) \text{~~and~~} F(z) := (K^\top y, -Kx).
\end{equation}
Given the special structure of $\theta$ and $F$ in \eqref{mVI-pd}, the golden ratio algorithm \eqref{gr-alg} breaks into
\begin{equation}\label{gr-alg2}
\left\{
\begin{array}{rclrcl}
  \bar{x}_n &=&  {\phi - 1 \over \phi} x_{n-1} + {1\over \phi} {\bar x}_{n-1},
  \;\; x_n &=& \prox_{\tau g}(\bar{x}_n - \tau K^\top y_{n-1}), \smallskip \\
  \bar{y}_n &=&  {\phi - 1 \over \phi} y_{n-1} + {1\over \phi} {\bar y}_{n-1},
  \;\; y_n &=& \prox_{\tau f^*}(\bar{y}_n + \tau Kx_{n-1}).
\end{array}
\right.
\end{equation}
Apparently, \eqref{gr-alg2} is a Jacobian type algorithm, which fails to fully utilize the latest available information.
Furthermore, the primal and dual step sizes are identically $\tau$, which offers less flexibility.
Given the nice convergence properties and the promising numerical results of golden ratio algorithm \cite{Malitsky2019Golden},
it is desirable to construct a Gauss-Seidel type golden ratio PDA for the saddle point problem \eqref{pd-prob}, which is able to fully take advantage of the problem structure. This motivated the current work.

%This feature inspired us to generalize this technique to the Arrow-Hurwicz method, and to present a new and efficient PDA with convergence in general. But we have to face two challenges:
%\bi
%\item (i) how to establish convergence and present acceleration when the convex combination is involved;
%\item (ii) for the widely-used case with special structure, how to further improve step sizes to get better numerical performance with the guarantee of convergence.
%\ei
%Interestingly by the basic variational theory, fixed point iteration and a balance between step's increase rate and size, we have achieved our goal.

\subsection{Contributions}
We adapt the convex combination technique \cite{Malitsky2019Golden} into the Arrow-Hurwicz scheme \eqref{pda_basic0} to
construct a Gauss-Seidel type golden ratio PDA (abbreviated as GRPDA), which not only preserves nice convergence properties but also performs favorably in practice.  Our main contributions are summarized below.
\bi
\item Using the convex combination technique  \cite{Malitsky2019Golden}, we construct a GRPDA with fixed step size parameters $\tau>0$ and $\sigma>0$. %
    Under the condition $\tau\sigma L^2<\psi\in (1, \phi]$, global convergence and  $\cO(1/N)$ ergodic rate of convergence measured by primal-dual gap function are established.
    Since $\tau\sigma L^2 > 1$ is permitted, GRPDA converges in a broader range of parameters compared to PDA, which typically requires $\tau\sigma L^2 < 1$.
    It is further explained that GRPDA is an ADMM-like method with a proximal term $\frac{1}{2}\|x-x_{n-1}\|_M^2$, where $M=\frac{1}{\tau}I-\sigma K^\top K$, plus an extra linear term. As a consequence of the extra linear term, the weighting matrix $M$ is allowed to be indefinite.
    %
    %This gives an explanation why GRPDA performs better than the classic PDA \cite{Chambolle2011A}.

\item When either $g$ or $f^*$ is strongly convex, an accelerated GRPDA is constructed to improve the  ergodic convergence rate from $\cO(1/N)$ to $\cO(1/N^2)$.

\item Let $b\in \R^p$. For regularized least-squares problem, i.e., $f(\cdot) = {1\over 2}\|\cdot - b\|^2$, or linear equality constrained problem, i.e., $f(\cdot) = \iota_{\{b\}}(\cdot)$, we show via spectral analysis that the permitted range of  $\psi$ can be extended from $(1,\phi]$ to $(1,2]$ and meanwhile a relaxation step can be taken as well, where the relaxation parameter lies in $(0,3/2)$. The analysis for these two special cases is based on the fixed point theory of averaged operators.

\item We carry out numerical experiments on LASSO, nonnegative least-squares and minimax matrix game problems, with comparisons to some state-of-the-art algorithms, to demonstrate the favorable performance of the proposed algorithms.
\ei

\subsection{Organization}
The rest of this paper is organized as follows.
In Section \ref{sec_preliminarries}, we summarize some useful facts and identities and define the primal-dual gap function used in subsequent analysis. The GRPDA with fixed step sizes is presented in Section \ref{sec_PDAU}, where its convergence and $\cO(1/N)$ ergodic  rate of convergence are established as well.
In Section \ref{sec_Acceleration}, we present an accelerated GRPDA that enjoys a faster $\cO(1/N^2)$ ergodic rate of convergence under the assumption that either $g$ or $f^*$ is strongly convex.
Section \ref{sec_special_case} is devoted to the analysis of two special cases, i.e., regularized least-squares problem and linear equality constrained problem.
Numerical results in comparison with some state-of-the-art relative algorithms on LASSO, nonnegative least-squares and minimax matrix game problems are given in Section \ref{sec_experiments} to demonstrate the efficiency of the proposed algorithms.
Finally, some concluding remarks are drawn in Section \ref{sec_conclusion}.

\section{Preliminaries}
\label{sec_preliminarries}

In this section, we summarize some useful facts and identities and define the primal-dual gap function, which will be useful in our analysis.

\begin{fact}\label{fact_proj}
For any extended real-valued closed proper convex function $h$  defined on an Euclidean space $\R^m $, $\lambda > 0$ and $x\in \R^m $,
it holds that  $p = \prox_{\lambda h}(x)$ if and only if
\ben
\langle p-x, y-p\rangle\geq \lambda \big( h(p)-h(y) \big),~~ \forall y\in \R^m.
\een
\end{fact}

\begin{fact}\label{fact_ab}
Let $\{a_n\}_{n\in\bN}$ and $\{b_n\}_{n\in\bN}$ be two nonnegative real sequences. If there exists an integer $N>0$ such that
$a_{n+1} \leq a_n-b_n$ for all $n > N$, then $\lim_{n\rightarrow \infty}a_n$ exists and $\lim_{n\rightarrow \infty} b_n = 0$.
\end{fact}

The proofs of Facts \ref{fact_proj} and \ref{fact_ab} are easily derived and thus are omitted.
The following elementary identities will be used in our analsis. %Let $\R^m $ be an Euclidean space.
For any $x, y, z\in \R^m $ and $\alpha\in\bR$, there hold
\be
\langle x-y, x-z\rangle&=& \frac{1}{2}\|x-y\|^2 +\frac{1}{2}\|x-z\|^2- \frac{1}{2}\|y-z\|^2,\label{id}\\
\|\alpha x+(1-\alpha)y\|^2&=& \alpha \|x\|^2+(1-\alpha)\|y\|^2-\alpha(1-\alpha)\|x-y\|^2.\label{id2}
\ee

We next define the primal-dual gap function that will be used in our analysis. Let $(\bar{x}, \bar{y})\in\cS$ be any saddle point of the primal-dual problem (\ref{pd-prob}). Then, there hold $K\bar{x} \in \partial f^*(\bar{y})$ and $-K^\top \bar{y}\in \partial g(\bar{x})$, which are equivalent to
\ben
\left\{
\begin{array}
{l}
P(x) := g(x)-g(\bar{x}) + \langle K^\top \bar{y}, x-\bar{x}\rangle \geq0, \quad ~\forall x\in \R^q, \smallskip \\
D(y) := f^*(y)-f^*(\bar{y}) - \langle K \bar{x}, y-\bar{y}\rangle \geq 0, \quad ~\forall y\in \R^p.
\end{array}
\right.
\een
The primal-dual gap function is defined by
\be\label{G}
G(x,y) := P(x)+D(y)\geq0,~\forall (x,y)\in \R^q \times \R^p.
\ee
This primal-dual gap function is also used in, e.g., \cite{Malitsky2018A,Chambolle2016ergodic}. Note that, for fixed $(\bar{x}, \bar{y})\in\cS$, $P(x)$ and $D(y)$, and thus $G(x,y)$, are convex. Although $P(\cdot)$, $D(\cdot)$ and $G(\cdot,\cdot)$ depend on $(\bar{x}, \bar{y})$, we do not indicate this dependence in our notation since it is always clear from the context.

\section{Golden ratio primal-dual algorithm}
\label{sec_PDAU}
In this section, we present our GRPDA with fixed step sizes and establish its convergence and $\cO(1/N)$ ergodic convergence rate.

\subsection{GRPDA with fixed step sizes}
Recall that  $\phi$ represents the golden ratio and $L = \|K\|$ is the operator norm of $K$.
Below, we introduce our GRPDA with fixed step sizes $\tau>0$ and $\sigma>0$.
\vskip5mm
\hrule\vskip2mm
\begin{algo}
[GRPDA with fixed step sizes]\label{alg-basic}
{~}\vskip 1pt {\rm
\begin{description}
\item[{\em Step 0.}] Let $\tau, \sigma >0$ and $\psi \in (1, \phi]$ be such that $\tau\sigma L^2 < \psi$.
Choose  $x_0\in \R^q$, $y_0\in \R^p$. Set $z_{0} = x_0$ and $n=1$.
\item[{\em Step 1.}]Compute
\be\label{GRPDA}
\left\{\ba{rcl}
 z_{n}&=&\frac{\psi-1}{\psi} x_{n-1} + \frac{1}{\psi}z_{n-1}, \smallskip \\
 x_{n}&=&\prox_{\tau g}(z_{n}-\tau K^\top y_{n-1}), \smallskip \\
 y_{n}&=&\prox_{\sigma f^*}(y_{n-1}+\sigma K x_{n}).
 \ea\right.\ee
\item[{\em Step 2.}] Set $n\leftarrow n + 1$ and return to Step 1. %\\
  \end{description}
}
\end{algo}
\vskip1mm\hrule\vskip5mm

The GRPDA \eqref{GRPDA} and the PDA \eqref{pda_basic} are quite similar, both of which are modifications of the Arrow-Hurwicz scheme (\ref{pda_basic0}). The difference is that PDA \eqref{pda_basic} adopts an inertial or extrapolation technique, while GRPDA \eqref{GRPDA} uses a convex combination $z_{n}=\frac{\psi-1}{\psi} x_{n-1} + \frac{1}{\psi}z_{n-1}$.
To our knowledge, this convex combination technique with $\psi$ upper bounded by the golden ratio $\phi$ was initially introduced by Malitsky \cite{Malitsky2019Golden} to solve monotone mixed VI problem \eqref{mVI}.  A direct adaptation of \cite[Eq. (10)]{Malitsky2019Golden} to the primal-dual problem \eqref{pd-prob} would give the algorithm \eqref{gr-alg2}, which is a Jacobian type algorithm and has less flexibility since the primal and the dual step sizes are required to be identical.
In contrast, GRPDA \eqref{GRPDA} is a combination of \cite{Malitsky2019Golden} with the Arrow-Hurwicz scheme \eqref{pda_basic0}, and it is of Gauss-Seidel type since it utilizes the latest available information. Compared with the direct adaptation \eqref{gr-alg2},  the primal and the dual step sizes in \eqref{GRPDA} are not necessarily identical, which offers more flexibility in practice. Moreover, compared with PDA \eqref{pda_basic}, which requires $\tau\sigma L^2 < 1$, the condition $\tau\sigma L^2 < \psi \in (1, \phi]$ required by GRPDA \eqref{GRPDA} permits larger step sizes. Apparently, GRPDA has the same per iteration cost as those of the Arrow-Hurwicz method and the PDA.

\subsection{Connection with ADMM and PDA}
It is well known that the classical PDA can be interpreted as a preconditioned or linearized ADMM with positive definition proximal term, see \cite{Chambolle2011A,Shefi2014Rate,Esser2010General}.
As the main difference between  GRPDA  and PDA  is that the extrapolation step in \eqref{pda_basic} is replaced by the convex combination in \eqref{GRPDA}, it is natural to connect GRPDA with the ADMM. To be specific, we define the augmented Lagrangian function of
$\min_{x,w}\{g(x)+f(w): Kx=w\}$, a reformulation of \eqref{primal}, as
\[\nonumber
{\cL}_{\sigma}(x,w,y) := g(x) + f(w) + \langle y, Kx - w\rangle + {\sigma \over 2}\|Kx-w\|^2,
\]
where $y\in \R^p$ is the Lagrange multiplier and $\sigma >0$ is the penalty parameter.
Then, by using the Moreau decomposition $y = \prox_{\lambda f}(y) +  \lambda  \prox_{{1\over\lambda} f^*}( {y \over \lambda})$ for any $\lambda > 0$ and $y\in \R^p$ and following \cite{Chambolle2011A,Shefi2014Rate,Esser2010General}, it is easy to show that GRPDA \eqref{GRPDA} is equivalent to
\be\label{ADMM}
\left\{\ba{rcl}
z_{n}&=&\frac{\psi-1}{\psi} x_{n-1} + {1\over \psi} z_{n-1}, \smallskip\\
x_{n}&=&\arg\min\limits_x\left\{ {\cL}_{\sigma}(x,w_{n-1},y_{n-1}) + \ell(x,z_n,x_{n-1},w_{n-1}) + \frac{1}{2}\|x-x_{n-1}\|^2_{M} \right\}, \smallskip\\
w_{n}&=&\arg\min\limits_w {\cL}_{\sigma}(x_n,w,y_{n-1}), \smallskip\\
 y_{n}&=&y_{n-1}+\sigma(Kx_n-w_n),
 \ea\right.
\ee
where $\ell(x,z_n,x_{n-1},w_{n-1})  = {1\over \tau} \langle x, z_n-x_{n-1} \rangle -  \sigma \langle Kx,   Kx_{n-1} - w_{n-1} \rangle$, $M= {1\over \tau} I - \sigma K^\top K$ and $w_0\in \R^p$ is arbitrary. Here the proximal term $\frac{1}{2}\|x-x_{n-1}\|^2_{M}$ is used to cancel out ${\sigma \over 2}\|Kx\|^2$ in the $x$-subproblem. Since, by following \cite{Chambolle2011A,Shefi2014Rate,Esser2010General}, the derivation of \eqref{ADMM} from \eqref{GRPDA} is standard,  we omit the details.
By discarding the linear term $\ell(x,z_n,x_{n-1},w_{n-1}) $ in \eqref{ADMM} and using the Moreau decomposition to carry out a similar reduction, we obtain
\ben%\label{pda-from-gr}
\left\{\ba{rcl}
{\bar y}_{n-1} &=& 2y_{n-1} - y_{n-2}, \smallskip\\
x_{n}&=& \prox_{\tau g}\big(x_{n-1} - \tau K^\top  {\bar y}_{n-1}\big), \smallskip\\
y_{n}&=& \prox_{\sigma f^*}(y_n + \sigma Kx_n),
 \ea\right.
\een
which is a form of PDA with $\delta = 1$.
Note that PDA requires $\tau\sigma L^2 < 1$ to guarantee the positive definiteness of $M$. In comparison, by introducing the linear term $\ell(x,z_n,x_{n-1},w_{n-1})$, GRPDA \eqref{ADMM} allows larger step sizes since $\tau\sigma L^2 < \psi \in (1,\phi]$ suffices for global convergence. In this case, $M$ is permitted to be indefinite. This broader convergence region of GRPDA is beneficial in practice.

\subsection{Convergence results}
We next establish the convergence and $\cO(1/N)$ ergodic convergence rate of GRPDA \eqref{GRPDA}. For convenience,  we let $\beta := \sigma/\tau$ in the rest of this section.  First, we present a useful lemma.

\begin{lem}\label{lem11}
Let $\{(x_n,y_n,z_n)\}_{n\in\bN}$ be the sequence generated by Algorithm \ref{alg-basic} from any initial point $(x_0, y_0)\in \R^q\times \R^p$ and $z_0=x_0$, and let $G(\cdot,\cdot)$ be defined as in (\ref{G}) with any fixed $(\bar{x}, \bar{y})\in \cS$.  Then, it holds for any $(x,y)\in \R^q\times \R^p$ that
\be
\nonumber
\tau G(x_n,y_n)
&\leq&\langle x_{n+1}-z_{n+1}, \bar{x}-x_{n+1}\rangle+ {1\over\beta}  \left\langle y_{n}-y_{n-1}, \bar{y}-y_{n}\right\rangle+\psi \left\langle x_{n}-z_{n+1}, x_{n+1}- x_n\right\rangle \\
\label{Lemma1-eq}
&&+ \tau \langle K^\top  (y_{n} - y_{n-1}), x_n-x_{n+1}\rangle.
\ee
\end{lem}
\proof
It follows from (\ref{GRPDA}) and Fact \ref{fact_proj} that
\be
\langle x_{n+1}-z_{n+1} + \tau K^\top y_{n}, ~ \bar{x}-x_{n+1}\rangle\geq \tau \big( g(x_{n+1})-g(\bar{x}) \big),\label{temp1_01}\\
\big\langle {1\over\beta} (y_{n}-y_{n-1}) - \tau K x_{n}, ~\bar{y}-y_{n}\big\rangle\geq \tau \big(f^*(y_{n})-f^*(\bar{y})\big).\label{temp1_001}
\ee
%Similarly as in (\ref{temp1_01}), for any $x\in X$ we have
%\ben
%\langle x_{n}-z_{n} + \tau K^\top y_{n-1}, ~~ x-x_{n}\rangle\geq \tau [g(x_{n})-g(x)].
%\een
%Substituting $x= x_{n+1}$ in the inequality above yields
Similar to (\ref{temp1_01}), it holds that
\be
\langle x_{n}-z_{n} + \tau K^\top y_{n-1},~ x_{n+1}-x_{n}\rangle\geq \tau\big(g(x_{n})-g(x_{n+1})\big).\label{tem11}
\ee
Elementary calculations show that the addition of (\ref{temp1_01}), (\ref{temp1_001}) and (\ref{tem11}) gives
\ben
&&\langle x_{n+1}-z_{n+1},~  \bar{x}-x_{n+1}\rangle+ {1\over\beta}  \big\langle  y_{n}-y_{n-1}, ~ \bar{y}-y_{n}\big\rangle+\left\langle x_{n}-z_{n}, ~ x_{n+1}- x_{n}\right\rangle \nonumber\\
&& + \tau \langle K^\top (y_{n}-y_{n-1}), ~ x_{n}-x_{n+1}\rangle-\tau \langle K^\top \bar{y}, ~   x_{n}-\bar{x}\rangle +\tau \langle K\bar{x},~ y_{n}-\bar{y}\rangle\nonumber\\
&\geq&\tau\big(f^*(y_{n})-f^*(\bar{y})\big)+\tau\big(g(x_{n})-g(\bar{x})\big),
\een
which is the same as \eqref{Lemma1-eq} by considering $x_n - z_n = \psi (x_n - z_{n+1})$ (implied by the first equality in \eqref{GRPDA}) and following the definition of $G(\cdot,\cdot)$ in (\ref{G}). This completes the proof.
\endproof

Now, we are ready to establish the convergence of GRPDA \eqref{GRPDA}.

\begin{thm}\label{thm11}
Let $\{(x_n,y_n,z_n)\}_{n\in\bN}$ be the sequence generated by Algorithm \ref{alg-basic} from any initial point $(x_0, y_0)\in \R^q\times \R^p$ and $z_0=x_0$, and let $G(\cdot,\cdot)$ be defined as in (\ref{G}) with any fixed $(\bar{x}, \bar{y})\in \cS$.  Then, $\{(x_{n}, y_n)\}_{n\in\bN}$ converges to a solution of (\ref{pd-prob}), i.e., an element of $\cS$.
\end{thm}
\proof
From Lemma \ref{lem11} and (\ref{id}), we obtain
\be
\nonumber &&\|x_{n+1}-\bar{x}\|^2+{1\over\beta} \|y_{n}-\bar{y}\|^2+ 2 \tau G(x_n,y_n)\\
\nonumber &\leq&\|z_{n+1}-\bar{x}\|^2 + {1\over\beta} \|y_{n-1}-\bar{y}\|^2  + 2\tau\langle K^\top (y_{n}-y_{n-1}),~ x_{n}-x_{n+1}\rangle\\
&&-\psi\|z_{n+1}-x_n\|^2+(\psi-1)\|x_{n+1}-z_{n+1}\|^2 -\psi\|x_{n+1}-x_n\|^2 -{1\over\beta} \|y_{n}-y_{n-1}\|^2. \label{thm3.1-eq1}
\ee
Since $x_{n+1}={\psi\over \psi-1} z_{n+2} - {1 \over \psi-1} z_{n+1}$, it follows from (\ref{id2}) that
\be\label{x-z-relation}
\|x_{n+1}-\bar{x}\|^2
&=& {\psi\over \psi-1} \|z_{n+2}-\bar{x}\|^2- {1 \over \psi-1} \|z_{n+1}-\bar{x}\|^2 + {\psi \over (\psi-1)^2} \|z_{n+2}-z_{n+1}\|^2\nonumber\\
&=& {\psi\over \psi-1} \|z_{n+2}-\bar{x}\|^2- {1 \over \psi-1} \|z_{n+1}-\bar{x}\|^2+\frac{1}{\psi}\|x_{n+1}-z_{n+1}\|^2,
\ee
where the second equality is due to $z_{n+2}-z_{n+1} = {\psi-1\over\psi} (x_{n+1} - z_{n+1})$.
By plugging \eqref{x-z-relation} into \eqref{thm3.1-eq1} and combining terms, we obtain
\be\label{ineq1}
&    &{\psi\over \psi-1} \|z_{n+2}-\bar{x}\|^2+{1\over\beta} \|y_{n}-\bar{y}\|^2 + 2 \tau G(x_n,y_n)\nonumber\\
&\leq& {\psi\over \psi-1} \|z_{n+1}-\bar{x}\|^2 + {1\over\beta} \|y_{n-1}-\bar{y}\|^2 + 2\tau\langle K^\top (y_{n}-y_{n-1}),~ x_{n}-x_{n+1}\rangle \nonumber\\
& & -\psi\|z_{n+1}-x_n\|^2 +(\psi-1-\frac{1}{\psi})\|x_{n+1}-z_{n+1}\|^2 -\psi\|x_{n+1}-x_n\|^2 -{1\over\beta} \|y_{n}-y_{n-1}\|^2 \nonumber\\
&\leq& {\psi\over \psi-1} \|z_{n+1}-\bar{x}\|^2 + {1\over\beta} \|y_{n-1}-\bar{y}\|^2 + 2\tau\langle K^\top (y_{n}-y_{n-1}),~ x_{n}-x_{n+1}\rangle \nonumber\\
& & -\psi\|z_{n+1}-x_n\|^2 - \psi\|x_{n+1}-x_n\|^2 -{1\over\beta} \|y_{n}-y_{n-1}\|^2,
\ee
where the second ``$\leq$" follows from $\psi-1-\frac{1}{\psi}\leq0$ as $\psi\leq\phi$.
Recall that $\beta = \sigma/\tau$. Then, $\tau\sigma L^2 < \psi$ is equivalent to  $\tau<  \sqrt{{\psi/\beta}}/L$.
Let $\tau=  \zeta  \sqrt{{\psi/\beta}}/L$ with $0<\zeta<1$. It follows from Cauchy-Schwarz inequality and $L = \|K^\top\|$ that
\ben
2\tau\langle K^\top  (y_{n}-y_{n-1}),~ x_{n}-x_{n+1}\rangle
&\leq&2\zeta \sqrt{{\psi/\beta}} \|y_{n}-y_{n-1}\|\|x_{n+1}-x_n\|\\
&\leq&\zeta \Big(\psi\|x_{n+1}-x_n\|^2 +  {1\over\beta} \|y_{n}-y_{n-1}\|^2\Big),
\een
which together with (\ref{ineq1})  leads to
\begin{eqnarray}\label{thm3.1-main}
  2 \tau G(x_n,y_n) + a_{n+1}  \leq a_n-b_n,
\end{eqnarray}
with
\begin{align}\label{anbn}
\left\{
\begin{array}{rcl}
a_n &=& {\psi \over \psi - 1}\|z_{n+1}-\bar{x}\|^2+{1\over\beta} \|y_{n-1}-\bar{y}\|^2, \smallskip\\
b_n &=& \psi\|z_{n+1}-x_n\|^2 +(1-\zeta)\big(\psi\|x_{n+1}-x_n\|^2+{1\over\beta} \|y_{n}-y_{n-1}\|^2\big).
\end{array}
\right.
\end{align}
Since $G(x_n,y_n)\geq0$, $a_{n}\geq0$ and $b_n \geq0$, it follows from Fact \ref{fact_ab} that $\lim_{n\rightarrow\infty}a_{n}$ exists and $\lim_{n\rightarrow\infty} b_n = 0$.
By the definition of $b_n$, this implies
\[\nonumber
\lim\limits_{n\rightarrow\infty}\|z_{n+1}-x_n\|
= \lim\limits_{n\rightarrow\infty}\|x_{n+1}-x_n\|
= \lim\limits_{n\rightarrow\infty}\|y_{n+1}-y_n\|=0.
\]
Furthermore, we also have $\lim\limits_{n\rightarrow\infty}\|x_{n} - z_{n}\|=0$ since $x_{n} - z_n = \psi (x_n - z_{n+1})$.

It follows from \eqref{anbn} and the existence of $\lim_{n\rightarrow\infty}a_{n}$ that the sequences $\{z_{n+1}\}_{n\in\bN}$ and $\{y_{n}\}_{n\in\bN}$, and thus
$\{(z_{n+1}, y_n)\}_{n\in\bN}$, are bounded. Since $\lim\limits_{n\rightarrow\infty}\|x_{n} - z_{n}\|=0$, $\{x_{n}\}_{n\in\bN}$ is also bounded.
Let $\{(x_{n_k+1},y_{n_k})\}_{k\in\bN}$ be a subsequence of $\{(z_{n+1}, y_n)\}_{n\in\bN}$, which is convergent to $(x^*,y^*)$.
Then $\lim_{n\rightarrow\infty}x_{n_k} = \lim_{n\rightarrow\infty}x_{n_k+1} = x^*$.
Similar to (\ref{temp1_01}) and (\ref{temp1_001}), for any $(x,y)\in \R^q\times \R^p$, there hold
\ben
\left\{\ba{l}
\langle x_{n_k+1}-z_{n_k+1} + \tau K^\top y_{n_k}, ~ x-x_{n_k+1}\rangle\geq \tau \big(g(x_{n_k+1})-g(x)\big), \smallskip \\
\big\langle {1\over\beta} (y_{n_k}-y_{n_k-1}) - \tau K x_{n_k}, ~y-y_{n_k}\big\rangle\geq \tau \big(f^*(y_{n_k})-f^*(y)\big).
\ea
\right.
\een
By driving $k\rightarrow\infty$, taking into account that both $g$ and $f^*$ are closed (and thus lower semicontinuous) and cancelling out $\tau>0$, we obtain
\ben
\langle  K^\top y^*, ~ x-x^*\rangle\geq  g(x^*)-g(x) \text{~~and~~} - \langle    K x^*, ~y-y^*\rangle\geq  f^*(y^*)-f^*(y),
\een
which hold for any $(x,y)\in \R^q\times \R^p$. This implies that $(x^*,y^*)\in \cS$.
Since all the discussions remain valid for any $(\bar{x},\bar{y})\in \cS$, $(\bar{x},\bar{y})$ can be replaced by $(x^*,y^*)$
in the definition of $\{a_n\}_{n\in \bN}$ in the first place. As such, we have $\lim_{k\rightarrow\infty} a_{n_k} = 0$ since
\begin{eqnarray*}
\lim_{k\rightarrow\infty} \|z_{n_k+1}-x^*\| = \lim_{k\rightarrow\infty} \|x_{n_k}-x^*\| = 0
\text{~~and~~}
\lim_{k\rightarrow\infty} \|y_{n_k-1}-x^*\| = \lim_{k\rightarrow\infty} \|y_{n_k}-x^*\| = 0.
\end{eqnarray*}
Since $\{a_n\}_{n\in\bN}$ is monotonically nonincreasing, it follows that $\lim_{n\rightarrow\infty} a_n = 0$ and thus
\ben
\lim_{n\rightarrow\infty}x_n = \lim_{n\rightarrow\infty}z_{n+1} = x^* \text{~~and~~} \lim_{n\rightarrow\infty}y_n = y^*.
\een
This completes the proof.
\endproof

\begin{rem}\label{rem-phi}
Our analysis is motivated by Malitsky \cite{Malitsky2019Golden}. In particular, to get rid of the term $\|x_{n+1}-z_{n+1}\|^2$ in (\ref{ineq1}), it is required that $\psi-1-{1\over \psi} \leq 0$, and apparently the golden ratio is an upper bound of such $\psi$. As long as $\tau\sigma L^2 < \psi$ is satisfied, larger $\psi$ allows larger step sizes $\tau$ and $\sigma$, which is helpful to improve numerical performance.
On the other hand, $\psi > 1$ is necessary to guarantee the nonnegativity of $\{a_n\}_{n\in\bN}$.
\end{rem}

\subsection{Ergodic convergence rate}
\label{sec_Rate}
For convex-concave saddle point problems, many algorithms exhibit $\cO(1/N)$ ergodic rate of convergence, see \cite{Chambolle2016ergodic,Monteiro2011Complexity,Nemirovski2006Prox-Method}. It is shown by Nemirovski \cite{Nemirovski2006Prox-Method} that this rate is in fact tight.
We next establish the same ergodic rate of convergence for GRPDA, where the measure is the primal-dual gap function defined in \eqref{G}.

\begin{thm}\label{thm2}
Let $\{(x_n,y_n,z_n)\}_{n\in\bN}$ be the sequence generated by Algorithm \ref{alg-basic} from any initial point $(x_0, y_0)\in \R^q\times \R^p$ and $z_0=x_0$, and let $G(\cdot,\cdot)$ be defined as in (\ref{G}) with any fixed $(\bar{x}, \bar{y})\in \cS$. Let $N\geq 1$ be any integer and define
\be\label{definition_XY0}
X_N=\frac{1}{N}\sum\nolimits_{n=1}^N x_n~~\mbox{and}~~Y_N=\frac{1}{N}\sum\nolimits_{n=1}^N y_n.
\ee
Then, it holds that
\ben
G(X_N,Y_N) \leq \frac{1}{2\tau N}\Big({\psi \over \psi-1} \|z_{2}-\bar{x}\|^2  + {1\over\beta}  \|y_{0}-\bar{y}\|^2\Big).
\een
\end{thm}
\proof
It follows from \eqref{thm3.1-main} and \eqref{anbn} that $2\tau G(x_n,y_n) \leq a_n - a_{n+1}$ for all $n\geq 1$.
%\ben
%2\tau G(x_n,y_n) \leq  {\psi \over \psi - 1} \left(\|z_{n+1}-\bar{x}\|^2-\|z_{n+2}-\bar{x}\|^2\right)
%+ {1\over\beta}  \left(\|y_{n-1}-\bar{y}\|^2 - \|y_{n}-\bar{y}\|^2\right).
%\een
Sum this up for $n= 1, \ldots, N$, we obtain $2\tau \sum_{n=1}^N G(x_n,y_n) \leq a_1 - a_{N+1} \leq a_1$.
Since $G(x,y) = P(x) + D(y)$ and $P(x)$ and $D(y)$ are, respectively, convex in $x$ and $y$, it follows from (\ref{definition_XY0}) that
\ben
G(X_N, Y_N) = P(X_N)+D(Y_N) \leq {1\over N} \sum_{n=1}^N P(x_n) + {1\over N} \sum_{n=1}^ND(y_n)  = {1\over N}\sum_{n=1}^N G(x_n,y_n).
\een
The conclusion follows since $a_1 = {\psi \over \psi-1} \|z_{2}-\bar{x}\|^2  + {1\over\beta}  \|y_{0}-\bar{y}\|^2$.
\endproof

\section{Accelerated GRPDA}
\label{sec_Acceleration}

When either $g$ or $f^*$ is strongly convex, it was shown in \cite{Chambolle2011A} that one can adaptively choose the primal and the dual step sizes, as well as the inertial parameter, so that PDA achieves a faster $\cO(1/N^2)$ convergence rate. Similar results have been achieved in \cite{Malitsky2018A} for PDA with variable step sizes from linesearch.
In this section, we show that the same is true for GRPDA provided that the parameters are chosen adaptively. Similar to \cite{Chambolle2011A,Malitsky2018A}, the ratio $\beta = \sigma/\tau$ plays an important role.

Under the regularity condition specified in Assumption \ref{asmp-1}, the ``$\min$" and the ``$\max$" in the primal dual problem (\ref{pd-prob}) can be switched \cite[Corollary 31.2.1]{Rockafellar1970Convex}, and thus \eqref{pd-prob} is equivalent to
\be\label{pd-prob2}
\min_{y\in \R^p} \max_{x\in \R^q}  ~~ f^*(y) - \langle K^\top y, x\rangle -g(x),
\ee
As such, by switching the roles of $(g,K,x,q)$ and $(f^*,-K^T,y,p)$, \eqref{pd-prob2} is reducible to (\ref{pd-prob}). Therefore, we will only
treat the case when $g$ is strongly convex for succinctness, and it is completely analogues when $f^*$ is strongly convex.
In this section, we thus make the following assumption.

\begin{asmp}
Assume that $g$ is $\gamma$-strongly convex, i.e., for some $\gamma>0$ it holds that
\ben
g(y) \geq g(x) + \langle u, y-x\rangle +
\frac{\gamma}{2}\|y-x\|^2,~~\forall \, x, y\in \R^q, \; \forall\, u\in \partial g(x).
\een
\end{asmp}

%
%For GRPDA with fixed step sizes, $\phi=\frac{1+\sqrt{5}}{2}$ is the largest constant $c$ from Remark \ref{rem-phi}, satisfying $\frac{1}{c}\geq c-1$ in order to get rid of the term $\|x_{n+1}-z_{n+1}\|^2$. When the step sizes $\tau_n$ are allowed to change, it needs satisfying
%\ben
%\frac{1}{c}\geq c\frac{\tau_n}{\tau_{n-1}}-1.
%\een
%We thus choose  $\psi\in(1,\phi)$, set $\varphi=\frac{1+\psi}{\psi^2}$ and require $\tau_n\leq\varphi\tau_{n-1}$ to get rid of the term $\|x_{n+1}-z_{n+1}\|^2$ and to derive an acceleration.

Our accelerated GRPDA is presented below. Recall that  $\phi$ denotes the golden ratio.

\vskip5mm
\hrule\vskip2mm
\begin{algo}
[Accelerated GRPDA when $g$ is $\gamma$-strongly convex]\label{alg-acc}
{~}\vskip 1pt {\rm
\begin{description}
\item[{\em Step 0.}]
  Let $\psi_0 = 1.3247...$ be the unique real root of $\psi^3-\psi-1 = 0$.
  Choose $\psi\in(\psi_0,\phi)$, $\beta_0>0$, $x_0\in \R^q$ and $y_0\in \R^p$, and set
  $z_0=x_0$, $\varphi=\frac{1+\psi}{\psi^2}$, $\tau_0=\frac{1}{L}\sqrt{{\psi\over\beta_0}}$ and $n=1$.
\item[{\em Step 1.}]Compute
\be
z_{n}&=& {\psi-1 \over \psi} x_{n-1} + \frac{1}{\psi}z_{n-1}, \label{GRPDA-acc-z}\\
x_{n}&=&\prox_{\tau_{n-1} g}(z_{n}-\tau_{n-1} K^\top y_{n-1}). \label{GRPDA-acc-x}
\ee
\item[{\em Step 2.}] Compute
 \be
 \omega_n&=&\frac{\psi-\varphi}{\psi+\varphi\gamma\tau_{n-1}},\label{GRPDA-acc-omega}\\
 \beta_{n}&=&\beta_{n-1}(1+\omega_n\gamma \tau_{n-1}), \label{GRPDA-acc-beta}\\
 \tau_n&=&\min\Big\{\varphi\tau_{n-1},~~\frac{\psi}{\tau_{n-1}\beta_nL^2}\Big\},\label{GRPDA-acc-tau}\\
 y_{n}&=&\prox_{\beta_n\tau_n f^*}(y_{n-1}+\beta_n\tau_n K x_{n}).\label{GRPDA-acc-y}
\ee
\item[{\em Step 3.}] Set $n\leftarrow n + 1$ and return to Step 1. %\\
  \end{description}
}
\end{algo}
\vskip1mm\hrule\vskip5mm

\begin{rem}
We present the following remarks on Algorithm \ref{alg-acc}.
\begin{enumerate}
  \item In Algorithm \ref{alg-acc}, $\omega_n$ is used to update $\beta_n$, which plays a key role in establishing the $\cO(1/N^2)$ convergence rate. The condition $\psi > \psi_0$, where $\psi_0$ is the unique real root of $\psi^3-\psi-1=0$, ensures that $\omega_n>0$.

  \item To achieve the $\cO(1/N^2)$ convergence rate,  a balance between the increase rate  $\tau_n/\tau_{n-1}$ and the size of $\tau_n$ is maintained in (\ref{GRPDA-acc-tau}).
      In particular, the ratio $\tau_n/\tau_{n-1} \leq \varphi\in (1,\psi)$ for $\psi \in (\psi_0, \phi)$.
      It follows from $\varphi=\frac{1+\psi}{\psi^2}$ that larger $\psi$ gives smaller bound $\varphi$, and vice versa.

  \item Assume $\gamma=0$. Then $\beta_n = \beta_0$ for all $n\geq 1$. Furthremore, it is easy to show by induction that, for all $n\geq 1$,
    the ``$\min$" in \eqref{GRPDA-acc-tau} is always attained by the second term on the right hand side and thus $\tau_n = \tau_0$ for all $n\geq 1$.
  Therefore, Algorithm \ref{alg-acc} reduces to Algorithm \ref{alg-basic}, i.e.,  GRPDA with fixed step sizes $\tau = \tau_0$ and $\sigma=\sigma_0:=\beta_0\tau_0$, in which case
      $\tau_0\sigma_0 L^2 = \psi < \phi$ mets the requirement of Algorithm \ref{alg-basic} as we require $\psi < \phi$ in Algorithm \ref{alg-acc}.

\end{enumerate}

\end{rem}

We next establish some useful properties of $\tau_n$, $\beta_n$ and $\omega_n$ generated in Algorithm \ref{alg-acc}.
For convenience, we let $\delta_n = {\tau_n \over \tau_{n-1}}$ in the rest of this section.

\begin{lem}\label{lem:parameters}
Let $\{(\omega_n, \beta_n, \tau_n)\}_{n\in\bN}$ be generated by Algorithm \ref{alg-acc} and define $\delta_n = {\tau_n \over \tau_{n-1}}$ for $n\geq 1$. Then, for all $n\geq 1$, there hold $0<\beta_{n-1} \leq \beta_{n}$, $\frac{\psi-\varphi}{\psi+\varphi\gamma\sqrt{\varphi}\tau_0}\leq\omega_n<1$  and
\be\label{tau_n_properties}
     \frac{1}{\sqrt{\varphi(1+\gamma\sqrt{\varphi}\tau_0)}}\frac{\sqrt{\psi}}{L\sqrt{\beta_{n}}}
\leq \tau_{n}
\leq \frac{\sqrt{\delta_n\psi}}{L\sqrt{\beta_{n}}}
\leq \frac{\sqrt{\varphi\psi}}{L\sqrt{\beta_{n}}}
\leq \frac{\sqrt{\varphi\psi}}{L\sqrt{\beta_0}}
= \sqrt{\varphi} \tau_0.
\ee
Moreover, there exists a constant $c>0$ such that $\beta_n\geq c n^2$ for all $n\geq1$.
\end{lem}
\proof
Since $\beta_0>0$, $\omega_n > 0$ for $n\geq 1$ and $\tau_n > 0$ for $n\geq 0$, it is clear from \eqref{GRPDA-acc-beta} that $0<\beta_{n-1} \leq \beta_{n}$ for all $n\geq 1$. Recall that $\delta_n = {\tau_n \over \tau_{n-1}}$. It follows from (\ref{GRPDA-acc-tau}) that
$\tau_n^2 \beta_n
\leq \frac{\delta_n\psi}{L^2}
\leq \frac{\varphi\psi}{L^2}$.
Further considering $\beta_n\geq \beta_0>0$ for all $n\geq 1$ and $\tau_0 = \sqrt{\psi/\beta_0}/L$, we see that
all the relations in \eqref{tau_n_properties}, except the first one, follow.
Now, the relation $\frac{\psi-\varphi}{\psi+\varphi\gamma\sqrt{\varphi}\tau_0}\leq\omega_n$ follows
from \eqref{GRPDA-acc-omega} and $\tau_{n-1} \leq \sqrt{\varphi}\tau_0$ for all $n\geq 1$. On the other hand,
$\omega_n < 1$ is obvious. We next prove the first inequality in \eqref{tau_n_properties} by induction.

It  follows from $\tau_{0}=\frac{\sqrt{\psi}}{\sqrt{\beta_0}L}$ and $\varphi > 1$ that
$\tau_{0}\geq\frac{1}{\sqrt{\varphi(1+\gamma\sqrt{\varphi}\tau_0)}}\frac{\sqrt{\psi}}{L\sqrt{\beta_{0}}}$.
Now, suppose that $\tau_{k}\geq\frac{1}{\sqrt{\varphi(1+\gamma\sqrt{\varphi}\tau_0)}}\frac{\sqrt{\psi}}{L\sqrt{\beta_{k}}}$ for $0\leq k\leq n$. Then,   \eqref{GRPDA-acc-tau} implies that either
\ben
\tau_{n+1}=\varphi\tau_{n}\geq\tau_{n}\geq\frac{1}{\sqrt{\varphi(1+\gamma\sqrt{\varphi}\tau_0)}} \frac{\sqrt{\psi}}{L\sqrt{\beta_{n}}}\geq \frac{1}{\sqrt{\varphi(1+\gamma\sqrt{\varphi}\tau_0)}}\frac{\sqrt{\psi}}{L\sqrt{\beta_{n+1}}}
\een
since $\beta_{n+1}\geq\beta_n$ and $\varphi>1$, or
\ben
\tau_{n+1}
=  \frac{\psi}{\tau_{n} \beta_{n+1}L^2}
&\geq& \frac{L\sqrt{\beta_{n}}}{\sqrt{\varphi\psi}} \frac{\psi}{\beta_{n+1}L^2}
=\sqrt{\frac{\beta_{n}}{\varphi\beta_{n+1}}} \frac{\sqrt{\psi}}{\sqrt{\beta_{n+1}}L}\\
&=&\frac{1}{\sqrt{{\varphi(1+\omega_{n+1}\gamma\tau_{n})}}}\frac{\sqrt{\psi}}{\sqrt{\beta_{n+1}}L}\\
&\geq&\frac{1}{\sqrt{\varphi(1+\gamma\sqrt{\varphi}\tau_0)}}\frac{\sqrt{\psi}}{\sqrt{\beta_{n+1}}L},
\een
where the first inequality follows from $\tau_{n}\leq \frac{\sqrt{\varphi\psi}}{L\sqrt{\beta_{n}}}$ and the second is due to
$\omega_{n+1} < 1$ and $\tau_n\leq\sqrt{\varphi}\tau_0$ for all $n\geq 0$. Therefore, the first inequality in \eqref{tau_n_properties} also holds for all $n\geq 1$.

To complete the proof, it only remains to show that there exists $c>0$ such that $\beta_n \geq c n^2$ for all $n\geq 1$.
For simplicity,  we let $\omega:=\frac{\psi-\varphi}{\psi+\varphi\gamma\sqrt{\varphi}\tau_0}$. Then,  $\omega_n \geq \omega$ for all $n\geq1$ and thus
\be\label{beta-1}
\beta_{n+1}
= \beta_{n}(1+\omega_{n+1}\gamma \tau_{n})
&\geq& \beta_{n}
\Big(1+ \frac{\omega\gamma}{\sqrt{\varphi(1+\gamma\sqrt{\varphi}\tau_0)}}\frac{\sqrt{\psi}}{L\sqrt{\beta_{n}}}\Big)
%\nonumber\\
%
%&=&
= \beta_{n}+  c_0 \sqrt{\beta_{n}},
\ee
where $c_0 := \frac{\omega\gamma \sqrt{\psi}}{\sqrt{\varphi(1+\gamma\sqrt{\varphi}\tau_0)}L} > 0$.
By induction, it is easy to show that $\beta_n\geq c n^2$ for all $n\geq1$ with $c := \min(c_0^2/9, \beta_1)>0$.
This completes the proof.
\endproof

Now, we are ready to establish the promised $\cO(1/N^2)$ ergodic convergence rate. % of Algorithm \ref{alg-acc}.

\begin{thm}\label{results-acc}
Let $\{(z_n,x_n,y_n,\omega_n, \beta_n, \tau_n)\}_{n\in\bN}$ be the sequence generated by Algorithm \ref{alg-acc} and $(\bar x, \bar y)\in\cS$ be any solution of \eqref{pd-prob}. For any integer $N\geq 1$, define
\ben%\label{x_n}
S_N=\sum_{n=1}^N \beta_{n}\tau_{n}, ~~
X_N=\frac{1}{S_N}\sum_{n=1}^N \beta_{n}\tau_{n}x_n \text{~~and~~}
Y_N=\frac{1}{S_N}\sum_{n=1}^N \beta_{n}\tau_{n}y_n.
\een
Then, there hold $\|z_{N+2}-\bar{x}\|=\cO(1/N)$ and $G(X_N,Y_N)=\cO(1/N^2)$.
\end{thm}
\proof
Fix $n\geq 1$.
Since $g$ is $\gamma$-strongly convex,  an inequality stronger than the one stated in Fact \ref{fact_proj} can be used.
It follows from \eqref{GRPDA-acc-x} that $z_{n}-\tau_{n-1} K^\top y_{n-1} - x_{n} \in \tau_{n-1} \partial  g(x_{n})$, and thus
\begin{eqnarray}\label{ineq-subgrad}
  \langle x_{n}-z_{n} + \tau_{n-1} K^\top y_{n-1}, ~x-x_{n}\rangle \geq \tau_{n-1} \Big(g(x_{n})-g(x)+\frac{\gamma}{2}\|x_{n}-x\|^2\Big), \; \forall\, x.
\end{eqnarray}
By passing $n+1$ to $n$ and $\bar x$ to $x$ in \eqref{ineq-subgrad}, we obtain
\begin{eqnarray}\label{thm4.1-ieq1}
  \langle x_{n+1}-z_{n+1} + \tau_n K^\top y_{n}, ~~ \bar{x}-x_{n+1}\rangle \geq \tau_n \Big(g(x_{n+1})-g(\bar{x})+\frac{\gamma}{2}\|x_{n+1}-\bar{x}\|^2\Big).
\end{eqnarray}
Similarly, by passing $x_{n+1}$ to $x$ in \eqref{ineq-subgrad} and multiplying both sides by $\delta_n = {\tau_n \over \tau_{n-1}}$, we obtain
\begin{eqnarray}\label{thm4.1-ieq2}
  \langle \delta_n(x_{n}-z_{n}) + \tau_{n} K^\top y_{n-1}, ~ x_{n+1}-x_{n}\rangle &\geq& \tau_{n} \Big(g(x_{n})-g(x_{n+1})+\frac{\gamma}{2}\|x_{n+1}-x_{n}\|^2\Big).
\end{eqnarray}
Similar to \eqref{temp1_001}, it follows from \eqref{GRPDA-acc-y} that
\begin{eqnarray}\label{thm4.1-ieq3}
\Big\langle {1\over \beta_n} (y_{n}-y_{n-1}) - \tau_n K x_{n}, ~\bar{y}-y_{n}\Big\rangle \geq \tau_n \big(f^*(y_{n})-f^*(\bar{y})\big)
\end{eqnarray}
From \eqref{GRPDA-acc-z}, it is easy to derive $x_n-z_n = \psi(x_n - z_{n+1})$.
Then, by adding \eqref{thm4.1-ieq1}-\eqref{thm4.1-ieq3} and using similar arguments as in Lemma \ref{lem11}, we obtain
\be\label{stronger-c}
\tau_n G(x_n,y_n)
&\leq&\langle x_{n+1}-z_{n+1}, \bar{x}-x_{n+1}\rangle+
\frac{1}{\beta_{n}} \big\langle y_{n}-y_{n-1}, \bar{y}-y_{n}\big\rangle
+\psi\delta_n \big\langle x_{n}-z_{n+1}, x_{n+1}- x_n\big\rangle \nonumber\\
&&+ \tau_n \langle K^\top (y_{n}- y_{n-1}), x_n-x_{n+1}\rangle - \gamma\tau_n  \|x_{n+1}-\bar{x}\|^2.
\ee
By using \eqref{id} and Cauchy-Schwartz inequality and combining terms, \eqref{stronger-c} implies
\ben
&&(1+2\gamma\tau_n)\|x_{n+1}-\bar{x}\|^2+\frac{1}{\beta_{n}}\|y_{n}-\bar{y}\|^2+ 2 \tau_n G(x_n,y_n)\nonumber\\
&\leq&\|z_{n+1}-\bar{x}\|^2+\frac{1}{\beta_{n}}\|y_{n-1}-\bar{y}\|^2  -\psi\delta_n\|z_{n+1}-x_n\|^2+(\psi\delta_n-1)\|x_{n+1}-z_{n+1}\|^2 \nonumber\\
&& -\psi\delta_n\|x_{n+1}-x_n\|^2-\frac{1}{\beta_{n}}\|y_{n}-y_{n-1}\|^2  + 2\tau_n\|K^\top (y_{n}- y_{n-1})\|\|x_{n+1}-x_n\|.
\een
Plugging in (\ref{x-z-relation}) and recalling $L = \|K\| = \|K^\top\|$, we obtain
\be\label{ineq12}
&&(1+2\gamma\tau_n)\frac{\psi}{\psi-1}\|z_{n+2}-\bar{x}\|^2
+\frac{1}{\beta_{n}}  \|y_{n}-\bar{y}\|^2+ 2 \tau_n G(x_n,y_n)\nonumber\\
&\leq&\frac{\psi+2\gamma\tau_n}{\psi-1}\|z_{n+1}-\bar{x}\|^2+\frac{1}{\beta_{n}}\|y_{n-1}-\bar{y}\|^2 +\Big(\psi\delta_n-1-\frac{1+2\gamma\tau_n}{\psi}\Big)\|x_{n+1}-z_{n+1}\|^2\nonumber\\
&&-\psi\delta_{n}\|z_{n+1}-x_n\|^2
-\psi\delta_{n}\|x_{n+1}-x_n\|^2-\frac{1}{\beta_{n}}\|y_{n}-y_{n-1}\|^2   \nonumber\\
&&+ 2\tau_n L \|y_{n}-y_{n-1}\|\|x_{n+1}-x_n\|.
\ee
It follows from $\tau_n\leq\frac{\sqrt{\psi\delta_n}}{L\sqrt{\beta_n}}$ (the second inequality in \eqref{tau_n_properties}) that
\ben
2\tau_n L \|y_{n}-y_{n-1}\|\|x_{n+1}-x_n\| % &\leq& 2\frac{\sqrt{\psi\delta_{n}}}{\sqrt{\beta_{n}}} \|y_{n}-y_{n-1}\|\|x_{n+1}-x_n\|\\
&\leq&\psi\delta_{n}\|x_{n+1}-x_n\|^2+ \frac{1}{\beta_{n}}\|y_{n}-y_{n-1}\|^2.
\een
Furthermore,  $\psi\delta_n-1-\frac{1+2\gamma\tau_n}{\psi} \leq \psi \varphi -1- \frac{1}{\psi} =0$
since $\delta_n\leq\varphi$.
Therefore, (\ref{ineq12}) implies
\be\label{ineq13}
&&(1+{2}\gamma\tau_n)\frac{\psi}{\psi-1}\|z_{n+2}-\bar{x}\|^2
+\frac{1}{\beta_{n}}  \|y_{n}-\bar{y}\|^2+ 2 \tau_n G(x_n,y_n)\nonumber\\
&\leq&\frac{\psi+{2}\gamma\tau_n}{\psi-1}\|z_{n+1}-\bar{x}\|^2+\frac{1}{\beta_{n}}\|y_{n-1}-\bar{y}\|^2. %-\psi\delta_{n}\|z_{n+1}-x_n\|^2.
\ee
Since $(1+{2}\gamma\tau_n)\frac{\psi}{\psi-1} = \frac{\psi(1+{2}\gamma\tau_n)}{\psi+{2}\gamma\tau_{n+1}} \frac{\psi+{2}\gamma\tau_{n+1}}{\psi-1}$ and
\ben
  \frac{\psi(1+{2}\gamma\tau_n)}{\psi+{2}\gamma\tau_{n+1}}
&\geq& \frac{\psi(1+{2}\gamma\tau_n)}{\psi+{2}\gamma\varphi\tau_{n}}
= 1+\frac{\psi-\varphi}{{\psi\over 2}+\gamma\varphi\tau_n}\gamma\tau_n
\geq 1+\frac{\psi-\varphi}{\psi+ \gamma\varphi\tau_n}\gamma\tau_n
= 1+\omega_{n+1}\gamma\tau_n,
\een
where the first inequality follows from $\tau_{n+1}\leq\varphi\tau_{n}$ and the second is due to $\psi^3 - \psi - 1>0$, it follows that
\be\label{key-ineq}
(1+{2}\gamma\tau_n)\frac{\psi}{\psi-1}
%= \frac{\psi(1+{2}\gamma\tau_n)}{\psi+{2}\gamma\tau_{n+1}} \frac{\psi+{2}\gamma\tau_{n+1}}{\psi-1}
%
\geq(1+\omega_{n+1}\gamma\tau_n)\frac{\psi+2\gamma\tau_{n+1}}{\psi-1}
=\frac{\beta_{n+1}}{\beta_n}\frac{\psi+2\gamma\tau_{n+1}}{\psi-1}.
\ee
Define $A_{n}:=\frac{\psi+2\gamma\tau_n}{2(\psi-1)}\|z_{n+1}-\bar{x}\|^2+ \frac{1}{2\beta_{n}}\|y_{n-1}-\bar{y}\|^2$.
Combining (\ref{ineq13}) and \eqref{key-ineq}, we deduce
\be\label{key-ineq2}
\beta_{n+1}A_{n+1}+\beta_{n}\tau_{n} G(x_n,y_n)
\leq \beta_{n}A_{n}.
\ee
By summing \eqref{key-ineq2} for $n = 1, \ldots, N$, we obtain
\be\label{beta-A}
\beta_{N+1}A_{N+1}+\sum_{n=1}^N\beta_{n}\tau_{n} G(x_n,y_n)
\leq \beta_{1}A_{1}.
\ee
%\ben
% \sum_{n=1}^N\beta_{n}\tau_{n} G(x_n,y_n)&=&\sum_{n=1}^N \beta_{n}\tau_{n}P(x_n)+\sum_{n=1}^N\beta_{n}\tau_{n}D(y_n)
%\geq S_N[P(X_N)+D(Y_N)]=S_NG(X_N,Y_N),
%\een
The convexity of $G(x,y)$, (\ref{beta-A}) and the definition of $A_n$ imply that
%the sequence $\{\|y_{n}-\bar{y}\|\}$ is bounded and
\be
G(X_N,Y_N)&\leq&  {1\over  S_N}\sum_{n=1}^N\beta_{n}\tau_{n} G(x_n,y_n) \leq { \beta_{1}A_{1} \over S_N},\label{G-XY}\\
\|z_{N+2}-\bar{x}\|^2&\leq& \frac{2(\psi-1)}{\psi + 2 \gamma \tau_{N+1}}
{\beta_{1}A_{1} \over \beta_{N+1}} \leq {2\beta_{1}A_{1} \over \beta_{N+1}}.
\label{z-N}
\ee
From Lemma \ref{lem:parameters}, there exists $c>0$ such that $\beta_n \geq c n^2$ for all $n\geq 1$.
Then,  (\ref{z-N}) gives
\ben
\|z_{N+2}-\bar{x}\|&\leq&\frac{\sqrt{2\beta_1A_1/c}}{N+1}.
\een
It follows from (\ref{beta-1}) that $\beta_{n+1}-\beta_{n}\geq
c_0 \sqrt{c}n$, which together with \eqref{GRPDA-acc-beta} and $\omega_{n+1}<1$ implies
\[\nonumber
\beta_n\tau_n=\frac{\beta_{n+1}-\beta_{n}}{\omega_{n+1}\gamma} \geq \frac{c_0 \sqrt{c}n}{\gamma}.
\]
As a result, $S_N =\sum_{n=1}^N \beta_{n}\tau_{n}=\cO(N^2)$, and thus $G(X_N,Y_N)=\cO(1/N^2)$ follows from (\ref{G-XY}).
\endproof

\section{GRPDA for two special cases}
\label{sec_special_case}
This section is devoted to GRPDA for two special cases of the structured convex optimization problem (\ref{primal}), which are specified in the following assumption.
\begin{asmp}\label{asmp-f}
Let $b\in\R^p$ be given.
Assume that either $f(\cdot) = \iota_{\{b\}}(\cdot)$, the indicator function of the singleton $\{b\}$, or
$f(\cdot) = {1\over 2}\|\cdot-b\|^2$.
\end{asmp}

The two cases specified in Assumption \ref{asmp-f} correspond respectively to
\ben
\min_{x\in \R^q} \{g(x): Kx = b \} \text{~~and~~} \min_{x\in \R^q} \frac 1 2 \|Kx-b\|^2 + g(x),
\een
which are abundant in practice. For example, linear inverse problems, which include many signal and image processing applications, usually enforce data fitting via $f(Kx)$, which takes the form $\frac 1 2\|Kx-b\|^2$ for noisy data and $\iota_{\{b\}}(Kx)$, or equivalently $Kx=b$, for ideal noiseless data.
A well known application is compressive signal/image sensing \cite{Dono2006Compressed,Needell2013compressed}, which recovers sparse or compressible signals from a small number of linear measurements via solving the basis pursuit and/or the LASSO problems \cite{Tibshirani1996Regression,Chen2001Atomic}.

The aim of this section is to propose a relaxed GRPDA when $f$ satisfies Assumption \ref{asmp-f}, which is a modification of Algorithm \ref{alg-basic} in two aspects: (i) extending the permitted range of $\psi$ from $(1,\phi]$ to $(1,2]$, and (ii) introducing a relaxation step.
In the rest of this section, we always assume that $f$ satisfies Assumption \ref{asmp-f}.
Then, for any $u\in\R^p$ and $\sigma>0$, it is easy to show that $\prox_{\sigma f^*}(u) =  \eta u+\varrho b$, where
\begin{eqnarray}
\label{def:eta}
\begin{array}
  {l}
%\prox_{\sigma f^*}(u) =  \eta u+\varrho b, \text{~~where~} u\in\R^p, \, \sigma >0, \text{~and} \smallskip \\
%
(\eta,\varrho) := (\eta(\sigma),\varrho(\sigma)) =
\left\{
  \begin{array}{ll}
    (1, -\sigma) & \hbox{\quad if $f(\cdot) = \iota_{\{b\}}(\cdot)$,} \smallskip \\
    {(1, \sigma) \over 1+\sigma} & \hbox{\quad if $f(\cdot) = {1\over 2}\|\cdot-b\|^2$.}
  \end{array}
\right.
\end{array}
\end{eqnarray}
In both cases, it holds that $\eta = \eta(\sigma) \in (0,1]$ for any $\sigma > 0$.

\subsection{GRPDA as fixed point iteration}
The iterative scheme GRPDA \eqref{GRPDA} applied to \eqref{primal} %with $f$ satisfying Assumption \ref{asmp-f}
appears as
\be\label{GRPDA-ls}
\left\{\ba{rcl}y_{n-1}&=&\eta(y_{n-2}+\sigma K x_{n-1})+\varrho b, \smallskip\\
z_{n}&=&\frac{\psi-1}{\psi} x_{n-1} +{1\over\psi} z_{n-1}, \smallskip\\
x_{n}&=&\prox_{\tau g}(z_{n}-\tau K^\top y_{n-1}),
\ea\right.
\ee
where $\eta$ and $\varrho$ are determined by \eqref{def:eta}.
Note that, since $f$ satisfies Assumption \ref{asmp-f},
the schemes \eqref{GRPDA-ls} and \eqref{GRPDA} are equivalent in the sense that they, if initialized properly, generalize exactly the same sequence of iterates.
Our choice of \eqref{GRPDA-ls}, which states explicitly the computing formulas of $(y_{n-1},z_n,x_n)$, instead of $(z_n,x_n,y_n)$ as in \eqref{GRPDA}, is mainly for convenience of analysis.

%For the special cases specified in Assumption \ref{asmp-f},
%We shall show that $\psi$, which is required to lie in $(1,\phi]$ in Algorithm \ref{alg-basic}, can be extended to the broader region $(1,2]$, and furthermore, a relaxation step can be taken. For this purpose,
%
To present our relaxed GRPDA, we first represent \eqref{GRPDA-ls} as a fixed point iterative scheme, which is summarized below.

%Recall that the stagnation of two vectors $u$ and $v$ is denoted by $(u; v)$.
\begin{lem}
Let $\xi_n=(z_n; \, x_n; \, y_{n-1})$. Then, the iterative scheme \eqref{GRPDA-ls} can be represented as the fixed point iteration
$\xi_{n}=\mathbb{S}\circ \mathbb{T}(\xi_{n-1})$, where
\begin{eqnarray}\label{def:FR}
\left\{
  \begin{array}{ccl}
  \mathbb{T}(\cdot) = T(\cdot) + \vartheta_1, & & \vartheta_1=(0; \, -\varrho\tau K^\top b; \, 0), \medskip \\
  \mathbb{S}(\cdot) = S(\cdot)+\vartheta_2, & & \vartheta_2=(0; \, 0^\top; \, \varrho b), \medskip \\
  S=\left[\ba{ccccc}I&&&0&0\\0&&&\prox_{\tau g}&0\\0&&&0&I\ea\right] & \text{~~and~~}& T=\left[\ba{ccccc}\frac{1}{\psi}I&~~&(1-\frac{1}{\psi})I&~~&0\\ \frac{1}{\psi}I&~~&(1-\frac{1}{\psi})I-\eta\tau\sigma K^\top K&~~&-\eta\tau K^\top \\ 0&~~&\eta\sigma K&~~&\eta I\ea\right].
  \end{array}
  \right.
\end{eqnarray}
\end{lem}
\begin{proof}
  Direct verification from \eqref{GRPDA-ls}.
\end{proof}

%\begin{lem}[{\hspace{-0.01cm}\cite[Proposition 4.2]{Bauschke2011Convex}}]\label{lem-app}
%Let $D$ be a nonempty subset of $H$ and let $T : D\rightarrow H$. Then
%the following are equivalent:\\
%(i) $T$ is firmly nonexpansive;\\
%(ii) $2T-I$ is nonexpansive.
%\end{lem}

Let $\bC$ be the set of complex numbers and $\Lambda(M)$ be the set of eigenvalues of a matrix $M$.
The following theorem is our key  to extend $\psi$ from $(1,\phi]$ to $(1,2]$ and to take a relaxation step.
\begin{thm}\label{lem-R}
Let $\tau, \sigma >0$ and $\psi \in (1, 2]$ be such that $\tau\sigma L^2 < \psi$ and $\eta \in (0,1]$.
Then, the matrix $T$, and thus the affine mapping $\mathbb{T}$, defined in (\ref{def:FR}) is firmly nonexpansive.
\end{thm}
\proof
It follows from, e.g., \cite[Proposition 4.2]{Bauschke2011book} that $T$ is firmly nonexpansive if and only if $2T-I$ is nonexpansive,
which is clearly equivalent to $|\lambda|\leq 1$ for any $\lambda \in \Lambda(2T-I)$.
Since $\eta \in (0,1]$, we have $2\eta - 1 \in (-1,1]$.
Therefore, to guarantee nonexpansiveness of $2T-I$, it is sufficient to show that $|\lambda|\leq 1$ for any $\lambda \in \Lambda(2T-I) \setminus \{\pm 1, 2\eta - 1\}$. The reason that we exclude $\pm 1$ and $2\eta-1$ from the spectrum $\Lambda(2T-I)$ will be clear below.

Let $\lambda \in \Lambda(2T-I) \setminus \{\pm 1, 2\eta - 1\}$ and $\xi= (z; \, x; \, y) \in \R^q \times \R^q \times \R^p$ be any eigen-pair of $2T-I$, i.e., $(2T-I)\xi=\lambda\xi$  and $\lambda \notin \{\pm 1, 2\eta - 1\}$. By the definition of $T$, $(2T-I)\xi=\lambda\xi$ appears as
\be
\Big(\frac{2}{\psi}-1\Big)z +\Big(2-\frac{2}{\psi}\Big)x=\lambda z,  \label{S1}\\
\frac{2}{\psi}z+ \Big(1-\frac{2}{\psi}\Big)x-2\eta\tau\sigma K^\top Kx-2\eta\tau K^\top y=\lambda x, \label{S2} \\
 2\eta\sigma Kx+ (2\eta-1)y=\lambda y.\label{S3}
\ee
%Note that $\frac{2}{\psi}z+\Big(1-\frac{2}{\psi}\Big)x=(\lambda+1) z-x$
Combining \eqref{S1} and \eqref{S2}, we obtain
\be\label{S33}
(\lambda+1)(z-x)-2\eta\tau\sigma K^\top Kx-2\eta\tau K^\top y=0.
\ee
%
%\textbf{Step 1.}
%If $\lambda=-1$,  the corresponding eigenvectors are $\xi=((1-\psi)x^\top,x^\top,y^\top)^\top$ with $x\in \R^q$ and $y=-\sigma K x$.
%%
%If $\lambda=2\eta-1$, the corresponding eigenvectors are $\xi=(\frac{\psi-1}{\psi\eta-1}x^\top,x^\top,y^\top)^\top$ with $x\in \null(K)$ (null space of $K$) and $K^\top y=\frac{1}{\tau}\frac{\psi(1-\eta)}{\psi\eta-1}x$.
%%
%If $\lambda=1$ and $\eta<1$, the corresponding eigenvectors are $\xi=(x^\top,x^\top,y^\top)^\top$ with $x\in \null(K^\top K)$, $y\in \null(K^\top )$ and  $y=\frac{\eta}{1-\eta}\sigma Kx$.
%
%
%\textbf{Step 2.}  We now consider the case of $\lambda\neq-1$, $\lambda\neq1$ and $\lambda\neq2\eta-1$.
%
%
First, we show that $KK^\top y\neq0$ by contradiction. Assume that $KK^\top y = 0$.
Then, it holds that $KK^\top Kx=0$ from (\ref{S3}), combining which with \eqref{S33} gives $(\lambda+1)K(z-x)=0$.
Since $\lambda\neq-1$, there must hold $Kz=Kx$. Then, multiplying both sides of (\ref{S2}) by $K$ shows that $(1-\lambda)Kx = 0$, which further implies that $Kx=0$ since $\lambda\neq1$. As a result, it follows from (\ref{S3}) and $\lambda\neq2\eta-1$ that there must hold $y=0$.
Thus, along with $Kx=0$ and $\lambda \neq -1$, (\ref{S33})  would lead to $x=z$.
Since $\lambda \neq 1$, it then follows from \eqref{S1} that $x=z=0$, which
contradicts to the fact that $\xi$ is an eigenvector and has to be nonzero.
Therefore, there must hold $KK^\top y\neq0$.

Since $KK^\top y\neq0$, we thus have $y\neq0$. For simplicity, we let
$\lambda_1:=\frac{\lambda+1}{2}$. Since $\lambda \notin \{\pm 1, 2\eta - 1\}$, we have
$\lambda_1 \notin \{0, 1, \eta\}$.
Then, (\ref{S1}) and (\ref{S3}) can be restated, respectively, as
\be
\Big(1-\frac{1}{\psi}\Big)x=\Big(\lambda_1-\frac{1}{\psi}\Big)z
\text{~~and~~}
\sigma Kx=\frac{\lambda_1-\eta}{\eta} y.\label{S45}
\ee
Multiplying both sides of (\ref{S3}) by $\tau K^\top $, adding to (\ref{S2}) and using \eqref{S1}, we obtain
$$(\lambda+1)(z-x-\tau K^\top y) = 0.$$
Since $\lambda\neq -1$,  we thus obtain $z-x=\tau K^\top y$, wihch together with (\ref{S45})  gives
\ben
    -\Big(\lambda_1-\frac{1}{\psi}\Big)\frac{\lambda_1-\eta}{\eta} y
&=& -\Big(\lambda_1-\frac{1}{\psi}\Big)\sigma Kx\nonumber\\
&=& \Big(\lambda_1-\frac{1}{\psi}\Big)\sigma K(\tau K^\top y-z)\nonumber\\
&=& \Big(\lambda_1-\frac{1}{\psi}\Big)\tau\sigma KK^\top y -  \sigma \Big(\lambda_1-\frac{1}{\psi}\Big) Kz \nonumber\\
&=& \Big(\lambda_1-\frac{1}{\psi}\Big)\tau\sigma KK^\top y- \sigma  \Big(1-\frac{1}{\psi}\Big) Kx \nonumber\\
&=& \Big(\lambda_1-\frac{1}{\psi}\Big)\tau\sigma KK^\top y - \Big(1-\frac{1}{\psi}\Big) \frac{\lambda_1 - \eta}{\eta} y.
\een
By combining the two terms of $y$ on both sides, we obtain
\ben
(\lambda_1-1)\frac{\eta-\lambda_1}{\eta} y&=& \Big(\lambda_1-\frac{1}{\psi}\Big) \tau\sigma KK^\top y,
\een
which simplifies to
\be\label{eig-y}
\frac{\eta-\lambda_1}{\eta} y=\frac{(\overline{\lambda_1}-1)\Big(\lambda_1-\frac{1}{\psi}\Big)}{|\lambda_1-1|^2}\tau\sigma KK^\top y. \label{S6}
\ee
Here  $\overline{\lambda_1}$ denotes the complex conjugate of $\lambda_1$.
Recall that $y\neq 0$. Then \eqref{eig-y} implies that $y$ is an eigenvector of $KK^\top$.
Since $KK^\top $ is real symmetric, positive semidefinite and $\|KK^\top\|=L^2$, its eigenvalues must be real and lie in $[0, L^2]$.
It is then implied by $KK^Ty\neq 0$ and $0<\tau\sigma L^2 < \psi$ that there exists $\psi_1\in(0,\psi)$ such that
$\tau\sigma KK^\top y=\psi_1y$.

Let $s, t\in\R$ and $\lambda = s + t \sqrt{-1}\in\bC$. Then $\lambda_1=\frac{1+s}{2} + \frac{t}{2}\sqrt{-1}$ and it is elementary to deduce from (\ref{S6})  that
\be
\frac{4\eta-2-2s}{\eta}&=&\frac{(s+1-\frac{2}{\psi})(s-1)+t^2}{|\lambda_1- 1|^2}\psi_1,\label{S7} \smallskip \\
\frac{t}{\eta}&=&\frac{t\Big(1-\frac{1}{\psi}\Big) \psi_1 }{|\lambda_1 -1|^2}.\label{S8}
\ee
We split the discussions in two cases, (i) $t=0$, and (ii) $t\neq 0$.
\bi
\item Case (i). If $t=0$, we have $s^2-2\eta(1-\psi_1)s-1+2\eta \varsigma=0$ from (\ref{S7}), where
$\varsigma := 1+\psi_1-\frac{2\psi_1}{\psi}$.
Let $u := \eta(1-\psi_1)$ and $v :=\eta^2(1-\psi_1)^2 -2\eta\varsigma + 1$. Then, we have $\lambda=s = u \pm \sqrt{v}$.
Since $s\in\bR$, there must hold $\eta^2(1-\psi_1)^2 -2\eta\varsigma  + 1 \geq0$, which, by
further considering $\eta \leq 1$, implies that
\ben
0<\eta\leq \left\{\ba{ll}\min\Big\{1,~~\frac{\varsigma -\sqrt{\varsigma^2-(1-\psi_1)^2}}{(1-\psi_1)^2}\Big\},&~~\psi_1\neq1, \medskip \\
\min\Big\{1,~~\frac{1}{4(1-\frac{1}{\psi})}\Big\},&~~\psi_1=1.\ea\right.
\een
By $\psi_1 \in (0, \psi)$ and $\psi>1$, it is elementary to show that $\sqrt{v} < 1 - u$ and $\sqrt{v} < 1 + u$.
Consequently, we have $\lambda \in (-1,1)$.

\item Case (ii). If $t\neq0$, we have $|\lambda_1-1|^2=\eta\psi_1\Big(1-\frac{1}{\psi}\Big)$ from (\ref{S8}). Then, (\ref{S7}) implies
\[\nonumber
(4\eta - 2 - 2s) \Big(1-{1\over \psi}\Big) = \Big(s + 1 - {2\over \psi}\Big) (s-1) + t^2,
\]
from which we obtain
\be\label{a-b}
\Big(s+1-\frac{2}{\psi}\Big)^2+t^2
&=& (4\eta - 2 - 2s) \Big(1-{1\over \psi}\Big) + \Big(s+1-{2\over \psi}\Big) \Big(2-{2\over \psi}\Big) \nonumber \\
&=& 4\Big(1-{1\over \psi}\Big) \Big(\eta -{1\over \psi}\Big)\nonumber \\
&=&\Big(1-\frac{2}{\psi}\Big)^2+4\eta\Big(1-\frac{1}{\psi}\Big)-1\nonumber\\
&\leq&\Big(1-\frac{2}{\psi}\Big)^2+4\Big(1-\frac{1}{\psi}\Big)-1\nonumber\\
&=&4\Big(1-\frac{1}{\psi}\Big)^2,
\ee
where the inequality follows from $\eta\in (0,1]$ and $\psi > 1$.
%Since $\psi\in(1,2]$, \eqref{a-b} and Figure \ref{Fig 0} show that $|\lambda| = \sqrt{s^2+t^2} \leq 1$.
%\begin{figure}[htp]
%\centering
%\begin{tikzpicture}
%  \fill[gray] (0.5,0)  circle(1.5);
%  \draw[->] (-2.5,0) -- (2.5,0) node[right] {$s$};
%  \draw[->] (0,-2.5) -- (0,2.5) node[above] {$t$};
%  \draw[](0,0)   circle (2);
%  \draw[](0.5,0)   circle(1.5);
%  \fill (0.5,0) node[below] {$\frac{2}{\psi}-1$} circle(1.5pt);
%  \fill (2,0) circle(1.5pt);
%  \fill (-2,0) circle(1.5pt);
%  \fill (2.3,0) node[above] {$1$};
%  \fill (-2.3,0) node[above] {$-1$};
%  \draw[<->](0.5,0.3) -- (2,0.3);
%  \draw[-](0.5,0.4) -- (0.5,0);
%  \draw[-](2,0.4) -- (2,0);
%  \fill (1.25,0.3) node[above] {$2-\frac{2}{\psi}$};
%\end{tikzpicture}
%\caption{The shaded region illustrates the range of $(s,t)$ satisfying (\ref{a-b}) with $\psi\in(1,2]$.}\label{Fig 0}
%\end{figure}
%
Since $\psi\in(1,2]$, a simple geometric argument based on \eqref{a-b} shows that $|\lambda| = \sqrt{s^2+t^2} \leq 1$.
\ei
Combining Cases (i) and (ii), we have shown that $|\lambda|\leq 1$  for any $\tau, \sigma > 0$ and $\psi\in (1,2]$ such that
$\tau\sigma L^2 < \psi$. This completes the proof.
\endproof

\subsection{Relaxed GRPDA}

Now, we are ready to present a relaxed GRPDA when $f$ satisfies Assumption \ref{asmp-f}.
The relaxation parameters lie in $(0,2/3)$ and the parameter $\psi$, which is restricted to $(1, \phi]$ in Algorithm \ref{alg-basic}, can now be extended to $(1, 2]$. The relaxed GRPDA is summarized below.

\vskip5mm
\hrule\vskip2mm
\begin{algo}
[Relaxed GRPDA]\label{algo-relax}
{~}\vskip 1pt {\rm
\begin{description}
\item[{\em Step 0.}]  Let $\tau, \sigma > 0$ and $\psi\in (1,2]$ be such that $\tau\sigma L^2 < \psi$, $(\eta,\varrho) = (\eta(\sigma),\varrho(\sigma))$ be defined in \eqref{def:eta}, and
$\{\rho_n\}_{n\in \bN} \subseteq (0,\frac{3}{2})$ be such that $\sum_{n\in\mathbb{N}}\rho_n(1-\frac{2\rho_n}{3})=+\infty$.
Choose $x_0\in \R^q$ and $y_{-1}\in \R^p$. Set $z_0 = x_0$ and $n=1$.
\item[{\em Step 1.}] Compute
\ben
\left\{\ba{rcl}
 \widetilde{y}_{n-1} &=&\eta(y_{n-2}+\sigma K x_{n-1})+\varrho b, \smallskip \\
 \widetilde{z}_{n}&=& \frac{\psi-1}{\psi}x_{n-1} + {1\over \psi}z_{n-1}, \smallskip \\
 \widetilde{x}_{n}&=&\prox_{\tau g}(\widetilde{z}_{n}-\tau K^\top \widetilde{y}_{n-1}).
 \ea\right.\een
\item[{\em Step 2.}] Take a relaxation step
\ben
\left\{\ba{rcl}
 y_{n-1}&=& y_{n-2}  + \rho_n (\widetilde{y}_{n-1} - y_{n-2}), \smallskip \\
 z_{n}&=& z_{n-1}  + \rho_n (\widetilde{z}_n - z_{n-1}), \smallskip \\
 x_{n}&=& x_{n-1}  + \rho_n (\widetilde{x}_n - x_{n-1}).
 \ea\right.\een
\item[{\em Step 3.}] Set $n\leftarrow n + 1$ and return to Step 1.\\
  \end{description}
}
\end{algo}
\vskip1mm\hrule\vskip5mm

The convergence of Algorithm \ref{algo-relax} is established in the following theorem.
The key of its proof is to observe, based on Theorem \ref{lem-R} and \cite[Proposition 4.32]{Bauschke2011book}, that
the operator $\mathbb{S}\circ \mathbb{T}$ is $2/3$-averaged\footnote{An operator $P$ is $\alpha$-averaged for some $\alpha \in (0,1)$ if there exists a nonexpansive operator $Q$ such that $P = (1-\alpha)I +\alpha Q$.}, and thus the convergence result follows from
\cite[Proposition 5.15]{Bauschke2011book}.

\begin{thm}\label{thm-rGRPDA}
Let $f$ satisfy Assumption \ref{asmp-f} and $\{(z_n,x_n,y_{n-1})\}_{n\in\bN}$ be the sequence generated by Algorithm \ref{algo-relax} from any initial point $(x_0, y_{-1})\in \R^q\times \R^p$ and $z_0=x_0$.
Then, $\{(x_{n}, y_n)\}_{n\in\bN}$ converges to a solution of (\ref{pd-prob}), i.e., an element in $\cS$.
\end{thm}
\begin{proof}
Let $\mathbb{S}$ and $\mathbb{T}$ be defined as in \eqref{def:FR}, $\mathbb{G} := \mathbb{S} \circ \mathbb{T}$, and $\xi_n = (z_n; x_n; y_{n-1})$ for $n\geq 0$.
Then, the sequence $\{(z_n,x_n,y_{n-1})\}_{n\in\bN}$ generated by Algorithm \ref{algo-relax} from $(z_0, x_0, y_{-1})$ satisfies $\xi_{n} = \xi_{n-1} + \rho_n (\mathbb{G}(\xi_{n-1})-\xi_{n-1})$ for $n\geq 1$. The key of the proof is to show that $\mathbb{G}$ is $2/3$-averaged, which we argue below.

Apparently, $\eta$ defined in \eqref{def:eta} lies in $(0,1]$. It then follows from Theorem \ref{lem-R} that $\mathbb{T}$ defined in \eqref{def:FR} is firmly nonexpansive under the condition of Algorithm \ref{algo-relax}, i.e., $\tau, \sigma > 0$ and $\tau\sigma L^2 < \psi \in (1,2]$. On the other hand, it is well known that the proximal operator $\prox_{\tau g}$ is firmly nonexpansive for any $\tau > 0$. Then, by the definition in \eqref{def:FR}, $\mathbb{S}$ is also firmly nonexpansive. Consequently, it follows from \cite[Proposition 4.32]{Bauschke2011book} that $\mathbb{G} = \mathbb{S}\circ \mathbb{T}$ is $2/3$-averaged.

Denote by $\mbox{Fix}(\mathbb{G})$ the set of fixed points of  $\mathbb{G}$, i.e.,
$\mbox{Fix}(\mathbb{G}) := \left\{ \xi\in \R^q\times \R^q\times \R^p:  \xi=\mathbb{G}(\xi)\right\}$.
By the construction of $\mathbb{S}$ and $\mathbb{T}$, it is elementary to show that
\[\nonumber
\mbox{Fix}(\mathbb{G}) = \{\xi = (z; \, x; \, y) \mid z = x, \; (x, y)\in\cS\},
\]
where $\cS$ is defined in \eqref{def:cS}.
By Assumption \ref{asmp-1}, $\cS$, and thus  $\mbox{Fix}(\mathbb{G})$, is nonempty.
Consequently, it follows from \cite[Proposition 5.15]{Bauschke2011book} that the sequence $\{\xi_n\}_{n\in\bN}$
is Fej\'{e}r monotone with respect to $\text{Fix}(\mathbb{G})$ and converges to a point in $\text{Fix}(\mathbb{G})$.
This completes the proof.
\end{proof}

%We summarize the above discussions in the following theorem.
%
%\begin{thm}[{\hspace{-0.01cm}\cite[Proposition 5.15]{Bauschke2011book}}]
%  Let $\mathbb{S}$ and $\mathbb{T}$ be defined in \eqref{def:FR}, $\tau, \sigma >0$ and $\psi \in (1, 2]$ satisfy $\tau\sigma L^2 < \psi$,
%  and $\eta$ be defined in \eqref{def:eta}.
%  Then, the composition $\mathbb{G} := \mathbb{S}\circ \mathbb{T}$ is $2/3$-averaged.
%  Under Assumption \ref{asmp-1}, $\mbox{Fix}(\mathbb{G})$ is nonempty.
%
%%
% Let $\{\rho_n\}_{n\in\bN} \subseteq (0,\frac{3}{2})$ satisfy $\sum_{n\in\mathbb{N}}\rho_n(1-\frac{2\rho_n}{3})=+\infty$.
%  Then, for any $\xi_0\in \R^q\times \R^q\times \R^p$,    the sequence $\{\xi_n\}_{n\in\bN}$ generated by
%  $\xi_{n} = \xi_{n-1} + \rho_n (\mathbb{G}(\xi_{n-1})-\xi_{n-1})$, $n=1,2,3,\ldots$, satisfies\\
%(i) $\{\xi_n\}$ is Fej\'{e}r monotone with respect to $\text{Fix}(\mathbb{G})$,\\
%(ii) $\{\mathbb{G}(\xi_n)-\xi_n\}$ converges to 0;\\
%(iii) $\{\xi_n\}_{n\in\bN}$ converges to a point in $\text{Fix}(\mathbb{G})$.
%\end{thm}

\section{Numerical Experiments}
\label{sec_experiments}

In this section, we present numerical results on LASSO \cite{Tibshirani1996Regression}, nonnegative least-squares and minimax matrix game problems to demonstrate the performance of Algorithm \ref{alg-basic} (GRPDA), the accelerated GRPDA given in Algorithm \ref{alg-acc} (A-GRPDA), and the relaxed GRPDA described in Algorithm \ref{algo-relax} (R-GRPDA).
All the experiments were performed within Python 3.8 running on a 64-bit Windows PC with an Intel(R) Core(TM) i5-4590 CPU@3.30 GHz and 8GB of RAM.
All the results presented in this section are reproducible by specifying the $seed$ of the random number generator in our code, which is available at \url{https://github.com/cxk9369010/Golden-Ratio-PDA}.

In this section, we let $\|\cdot\|=\|\cdot\|_2$ be the $\ell_2$-norm induced by the dot inner product.
For LASSO and nonnegative least-squares problems, the component function $f$ has the form $f(\cdot) = {1\over 2}\|\cdot-b\|^2$. Thus, Assumption \ref{asmp-f} is satisfied and the relaxed GRPDA described in Algorithm \ref{algo-relax} can be applied.
Furthermore, GRPDA with fixed step sizes (Algorithm \ref{alg-basic}) can adopt the larger value $\psi = 2$ as well (equivalent to setting $\psi = 2$ and $\rho_n\equiv 1$ in Algorithm \ref{algo-relax}).
Since $f^*$ (rather than $g$, as required by Algorithm \ref{alg-acc}) is strongly convex, we can switch  $(g,K,x,q)$ with $(f^*,-K^T,y,p)$ and reduce \eqref{pd-prob2} to \eqref{pd-prob}, so that $g$ is strongly convex.
As such, the accelerated GRPDA given in Algorithm \ref{alg-acc} is also applicable.

The algorithmic parameters are specified as follows. For GRPDA, we used fixed step sizes $\tau=\frac{\sqrt{\psi}}{\sqrt{\beta}L}$ and $\sigma =\beta\tau$. For the minimax matrix game problem, we set $\psi = 1.618$, while for LASSO and nonnegative least-squares problems, which satisfy Assumption \ref{asmp-f}, we set $\psi = 2$. For  A-GRPDA, we set $\psi=1.5$ and $\beta_0 = 1$, while for R-GRPDA we set $\tau=\frac{\sqrt{2}}{\sqrt{\beta}L}$, $\sigma = \beta\tau$ and $\rho_n\equiv1.49$.
%
%\bi
%\item GRPDA (Algorithm \ref{alg-basic}) with $\psi = 1.618$, $\tau=\frac{\sqrt{\psi}}{\sqrt{\beta}L}$ %or $\tau=\frac{\sqrt{2}}{\sqrt{\beta}L}$,
%and $\sigma =\beta\tau$.
%\item A-GRPDA (Algorithm \ref{alg-acc}) with $\psi=1.5$ and $\beta_0 = 1$.
%\item R-GRPDA (Algorithm \ref{algo-relax}) with $\tau=\frac{\sqrt{2}}{\sqrt{\beta}L}$, $\sigma = \beta\tau$ and $\rho_n\equiv1.49$.
%\ei
%
The following state-of-the-art algorithms are compared:
\begin{itemize}
\item PDA (as given in \eqref{pda_basic}) with $\tau=\frac{1}{\sqrt{\beta}L}$ and $\sigma=\sqrt{\beta}\frac{1}{L}$ and $\delta=1$.
\item PGM (proximal gradient method, e.g., \cite{Beck2009A}) with fixed step size $\alpha > 0$. For   LASSO and nonnegative least-squares problems, it holds that $f(Kx) = {1\over 2}\|Kx-b\|^2$, and the PGM appears as $x_{n} = \prox_{\alpha g}(x_{n-1} - \alpha  K^\top  (Kx_{n-1} - b))$. We set $\alpha = 1/\|K\|^2$.
\item FISTA (fast iterative shrinkage thresholding algorithm \cite{Beck2009A}) with fixed step size $\alpha > 0$.
For LASSO and nonnegative least-squares problems, FISTA appears as
\begin{eqnarray*}
\left\{
\begin{array}
  {rcl}
   t_n &=& \Big( 1 + \sqrt{1 + 4 t_{n-1}^2}\Big)/2, \\
   y_{n-1} &=& x_{n-1} + {t_{n-1} - 1 \over t_n} (x_{n-1} - x_{n-2}), \\
   x_{n} &=& \prox_{\alpha g}(y_{n-1} - \alpha  K^\top (Ky_{n-1} - b)),
\end{array}\right.
\end{eqnarray*}
  where $t_0 = 1$ and $x_{-1} = x_0$. The same as PGM, we set $\alpha = 1/\|K\|^2$.

\item GRAAL (golden ratio algorithm \cite{Malitsky2019Golden}, or as given in \eqref{gr-alg2}) with $\tau=\frac{\phi}{2L}$, where $\phi=1.618$.
\end{itemize}
The choice of $\beta$ for GRPDA, R-GRPDA and PDA will be specified later. For all algorithms, we used the same initial points.
For the three tested problems, the evaluation of the proximal point mappings is relatively cheap compared with matrix-vector multiplications.
As such, all the compared algorithms have the same dominant per iteration cost, i.e., two matrix-vector multiplications, we only compare the their performance in terms of iteration numbers.

\begin{prob}[LASSO]\label{pro_1}
The LASSO problem has the form $\min_x F(x):=\frac 1 2 \|Kx-b\|^2 + \mu  \|x\|_1$,
where $K\in \R^{p\times q}$ and $b\in \R^p$ are given, and $x\in \R^q$ is an unknown signal.
In the context of compressive sensing, $p$ is much smaller than $q$.
\end{prob}

Apparently, the above LASSO problem corresponds to \eqref{primal} with $f(u) = \frac 1 2 \|u-b\|^2$ and $g(x) = \mu \|x\|_1$.
%$f^*(y) = \frac 1 2 \|y\|^2 + \langle b,y\rangle= \frac 1 2 \|y+b\|^2 -\frac{1}{2}\|b\|^2$.
Thus, Assumption \ref{asmp-f} is satisfied and $\psi=2$ can be adopted in GRPDA.
The $seed$ of the random number generator was set to $1$. We generated $x^*\in \R^q$ randomly. Specifically, $s$ nonzero components of $x^*$ were determined uniformly at random, and their values were drawn from the uniform distribution in $[-10, 10]$.
The matrix $K\in \bR^{p\times q}$ is constructed as in \cite{Malitsky2018A} by one of the following ways:
\bi
\item[(i)] All entries of $K$ were generated independently from $\cN(0, 1)$, the normal distribution with mean $0$ and standard deviation $1$. %The $s$ nonzero entries of $x^*$ are drawn from the uniform distribution in $[-10, 10]$;
\item[(ii)] First, we generated a matrix $A\in\R^{p\times q}$, whose entries are independently drawn from $\cN(0,1)$. Then, for a scalar $v\in(0, 1)$ we constructed the matrix $K$ column by column as follows: $K_1 =  A_1/\sqrt{1-v^2}$ and $K_j = vK_{j-1} + A_j$, $j=2,\ldots,q$.
    Here $K_j$ and $A_j$ represent the $j$th column of $K$ and $A$, respectively.
    As $v$ increases, $K$ becomes more ill-conditioned. In this experiment, we tested $v = 0.5$ and $0.9$.
    %The ground truth $x^*$ was generated in the same way as in (i).  \comm{Is there any references for generating $A$ like this? If yes, maybe a reference?}\reply{Yes, it is the reference \cite{Malitsky2018A}.}
\ei
The additive noise $\nu\in \bR^p$ was generated from $\cN(0, 0.1)$.
Finally, we set $b = Kx^* +\nu$.

The step size ratio $\beta = \sigma/\tau$ was set to be $400$ as in \cite{Malitsky2018A}, which was determined based on numerical experience. All the algorithms were initialized at $x_0 = 0$ and $y_0=-b$.
For the LASSO problem, we first compared the performance of GRPDA with varying $\psi$. The results are given in Figure \ref{Fig psi}, where
the decreasing behavior of function value errors $F(x_n)-F^*$ as the algorithm proceeded is presented. Here $F^* = \inf_x F(x)$ was approximately computed by running our algorithms for sufficiently many iterations.
It can be seen from Figure \ref{Fig psi} that GRPDA with larger $\psi$ converges faster for all the tested three cases.
Thus, in our comparison with other algorithms, we set $\psi=2$. Besides, instead of terminating the algorithms with some stopping criteria, we ran all algorithms for a fixed number of iterations and examined their convergence behavior, see  Figure \ref{Fig 1}.
The parameters $p$, $q$, $s$ and $v$ are given in the captions of the figures.

\begin{figure}[htp]
\centering
\subfigure[Case (i).]{
\includegraphics[width=0.32\textwidth]{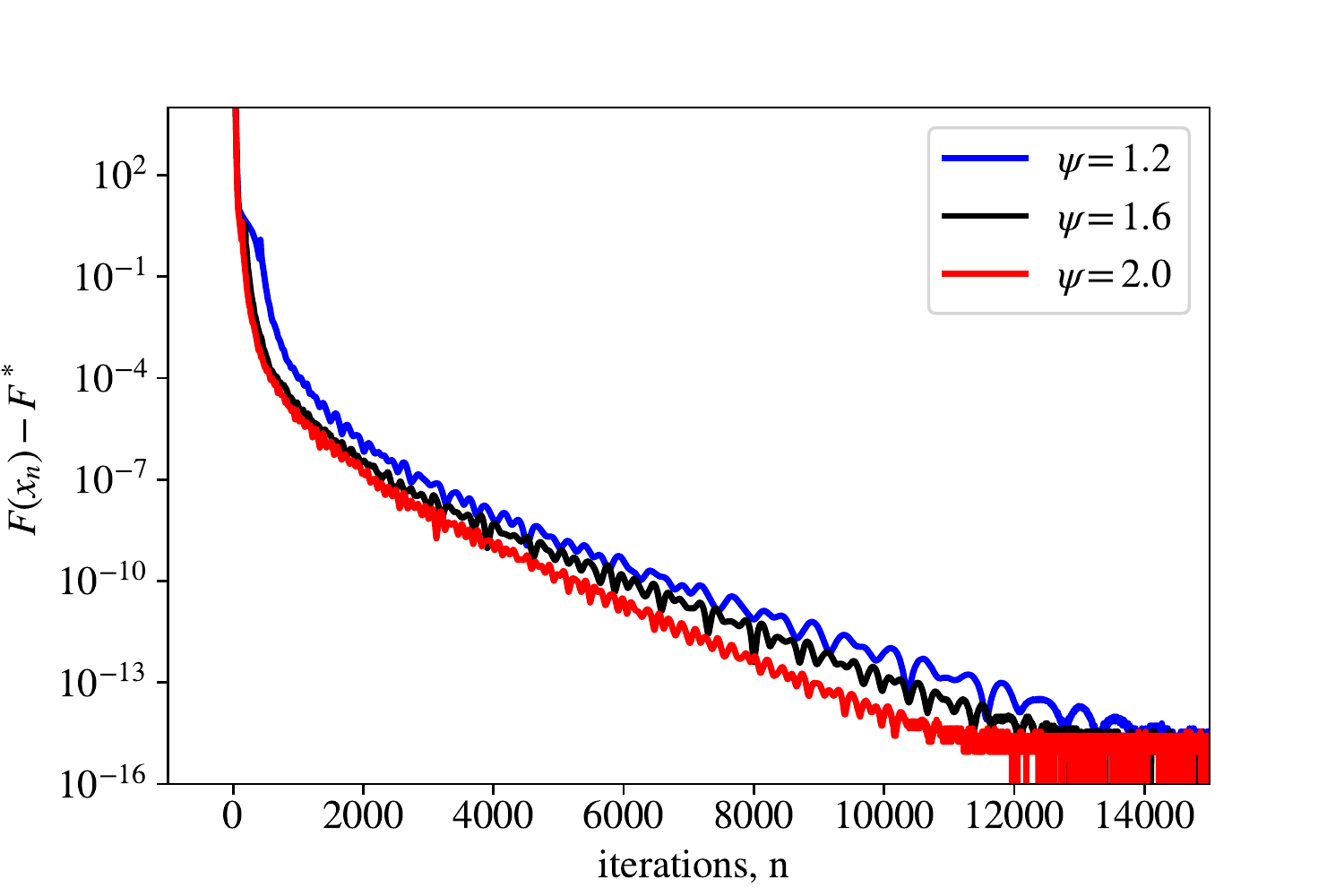}}
\subfigure[Case (ii) with $v=0.5$.]{
\includegraphics[width=0.32\textwidth]{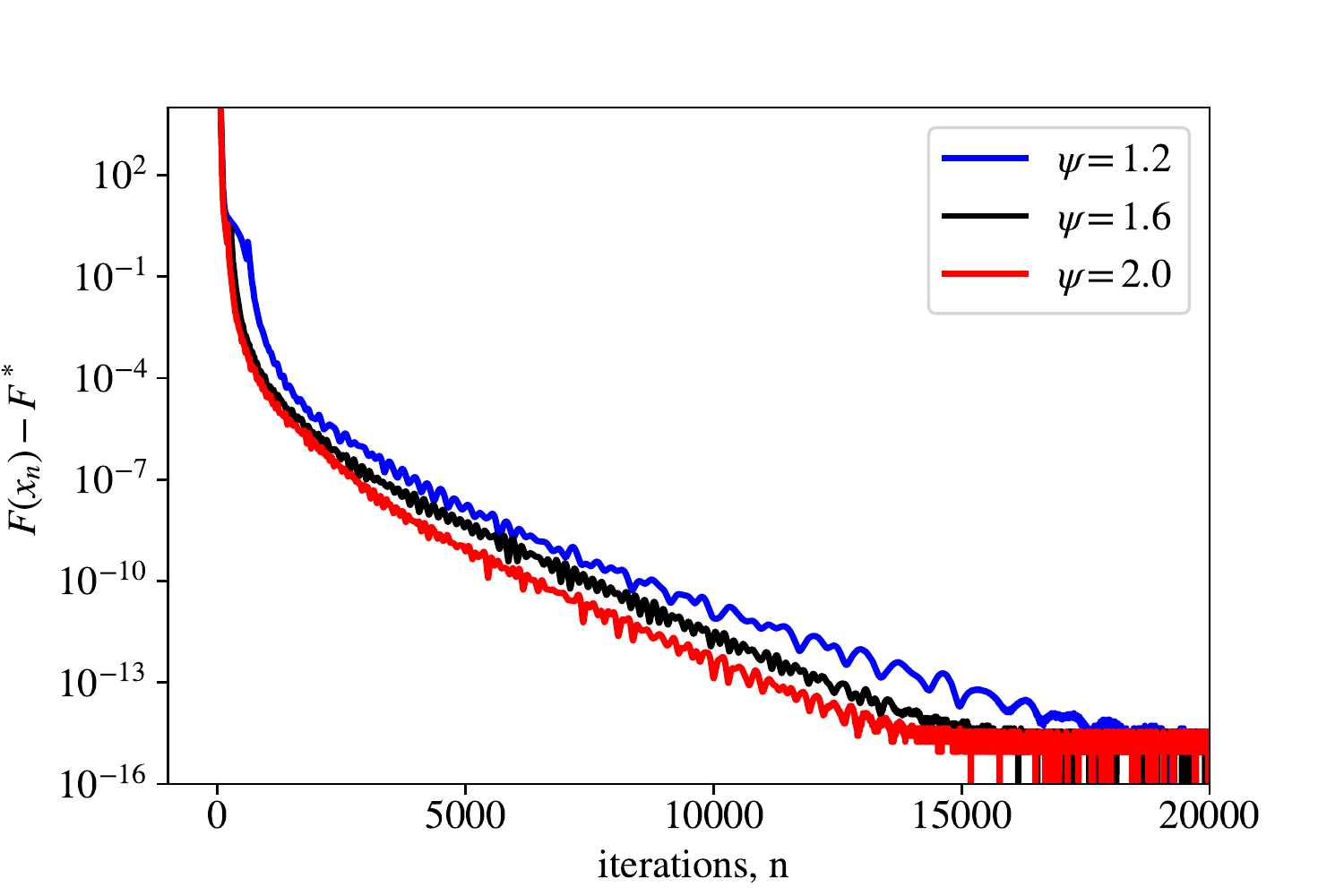}}
\subfigure[Case (ii) with $v=0.9$.]{
\includegraphics[width=0.32\textwidth]{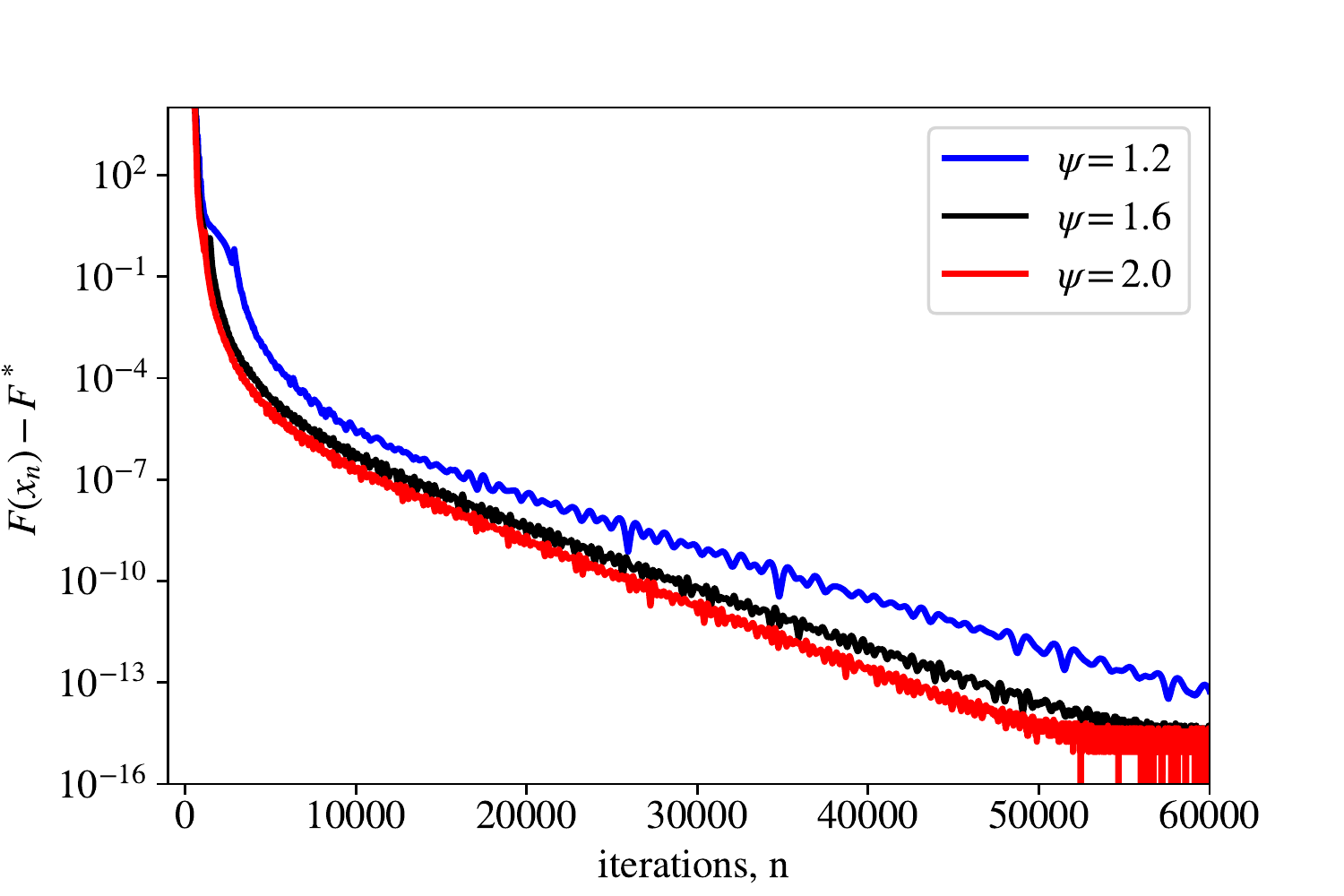}}
\caption{Evolution of $F(x_n)-F^*$ for the LASSO problem with different $\psi$ and $(p,q,s) = (200,1000,10)$.}
\label{Fig psi} %% label for entire figure
\end{figure}

It can be seen from Figure \ref{Fig 1} that primal-dual type methods outperform the PGM, FISTA and GRAAL. For all the three tested cases, GRPDA performs better than PDA, R-GRPDA is even better and A-GRPDA is the best. In particular, the proposed Gauss-Seidel type golden ratio algorithms, including GRPDA, A-GRPDA and R-GRPDA, perform much better than the Jacobian type golden ratio algorithm GRAAL as given in \eqref{gr-alg2}.

\begin{figure}[htp]
\centering
\subfigure[Case (i) with $(p,q,s) = (200,1000,10)$.]{
\includegraphics[width=0.45\textwidth]{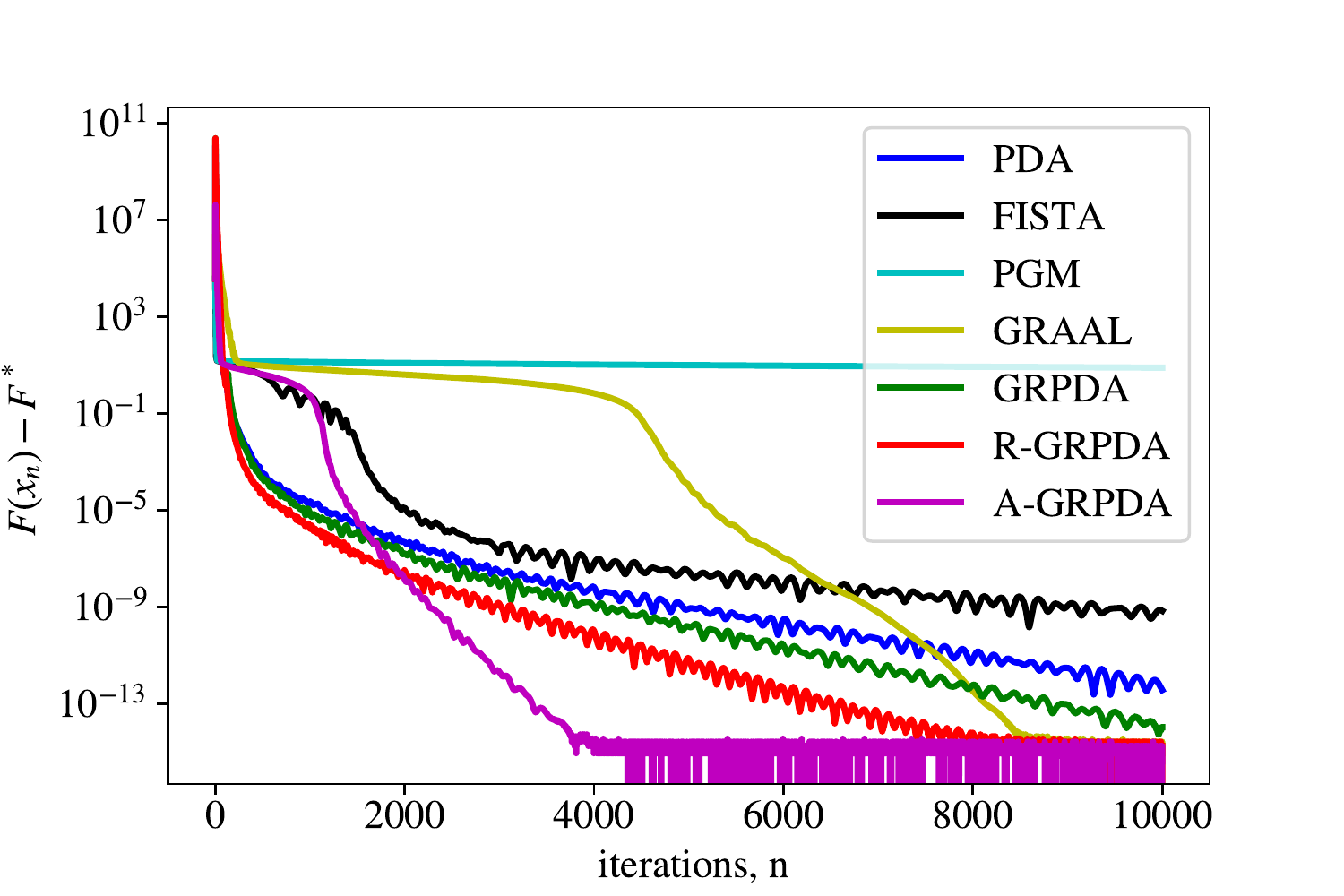}}
\subfigure[Case (i) with $(p,q,s) = (1000,2000,100)$.]{
\includegraphics[width=0.45\textwidth]{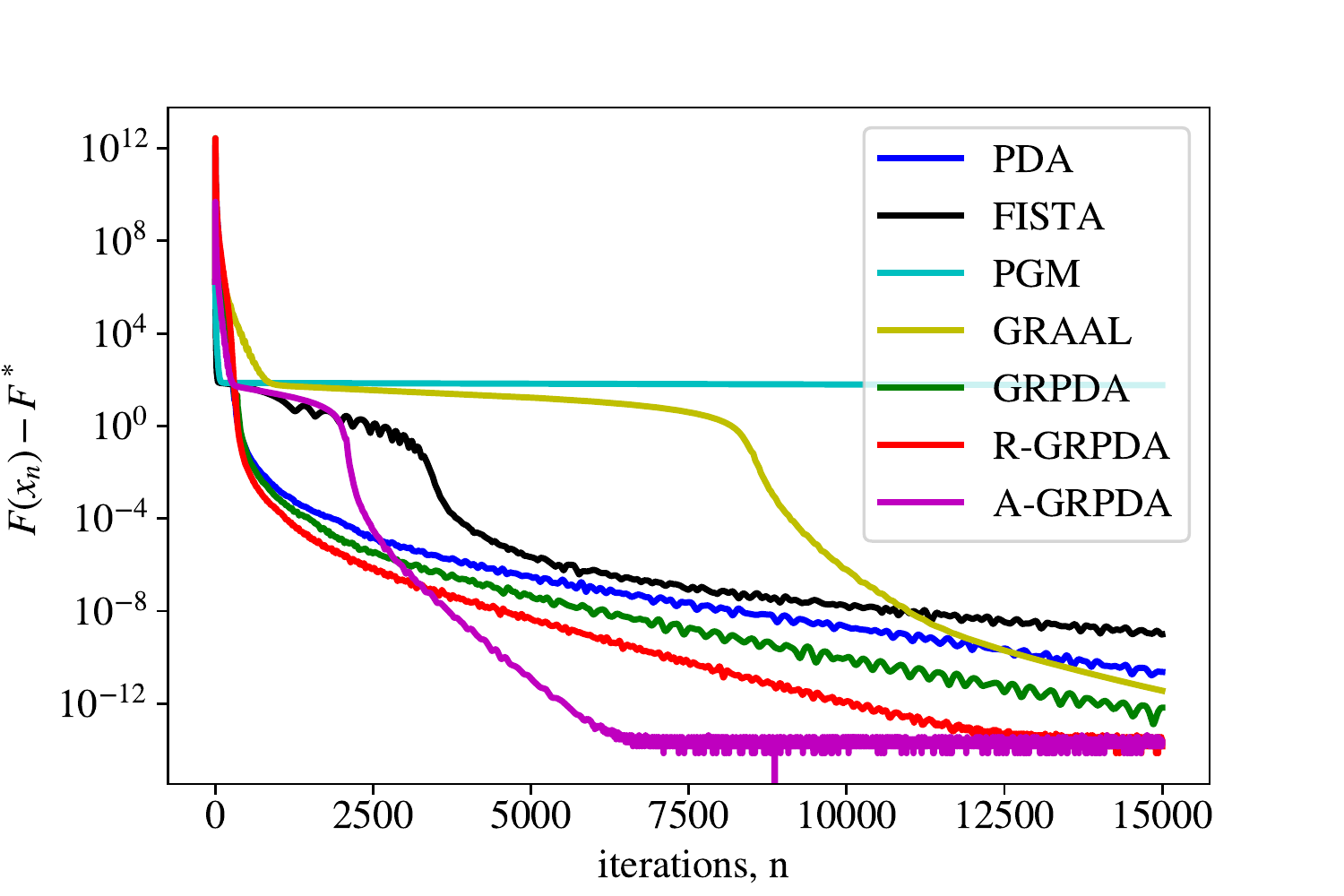}}
\subfigure[Case (ii) with $(p,q,s,v) = (1000,5000,100,0.5)$.]{
\includegraphics[width=0.45\textwidth]{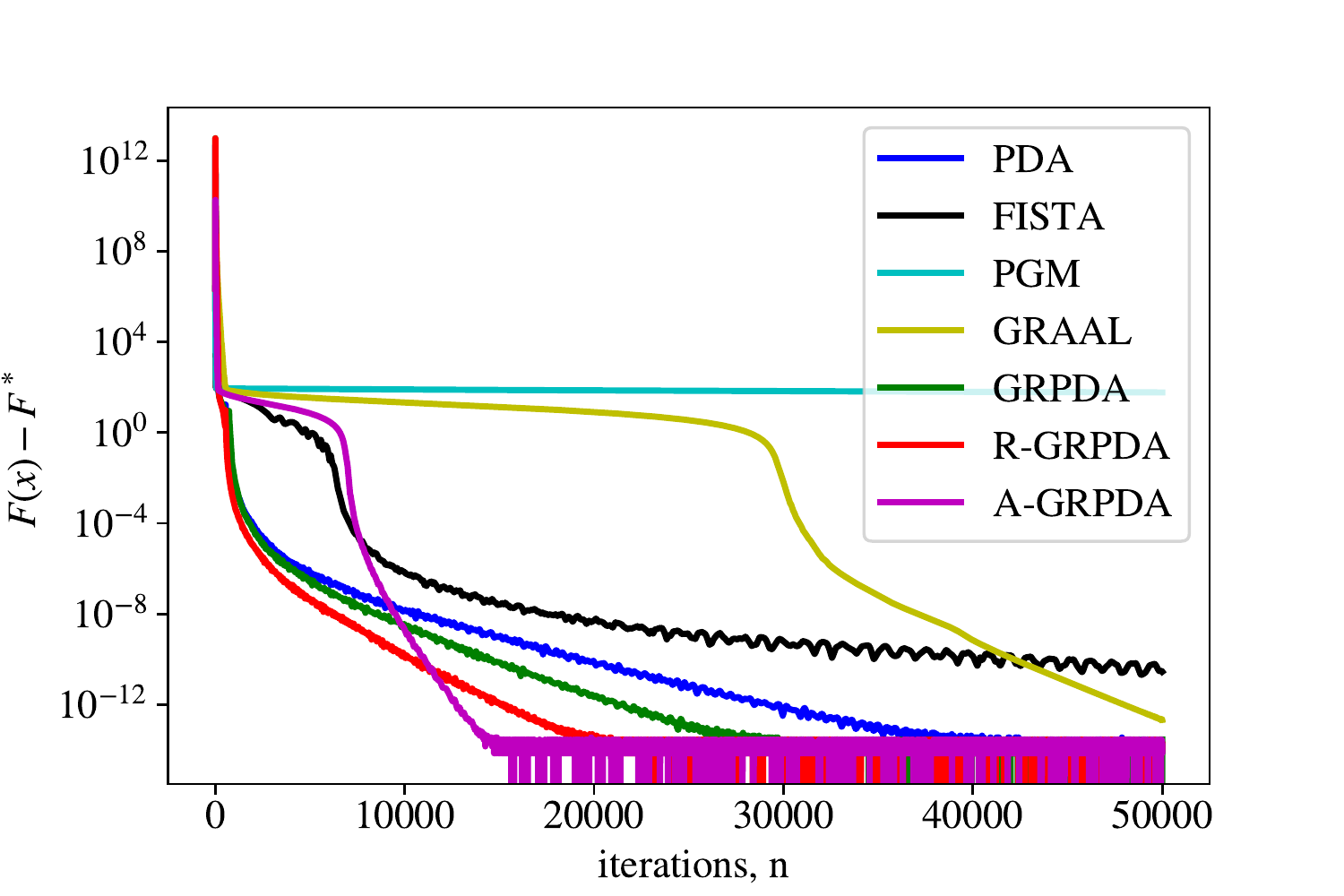}}
\subfigure[Case (ii) with $(p,q,s,v) = (1000,5000,100,0.9)$.]{
\includegraphics[width=0.45\textwidth]{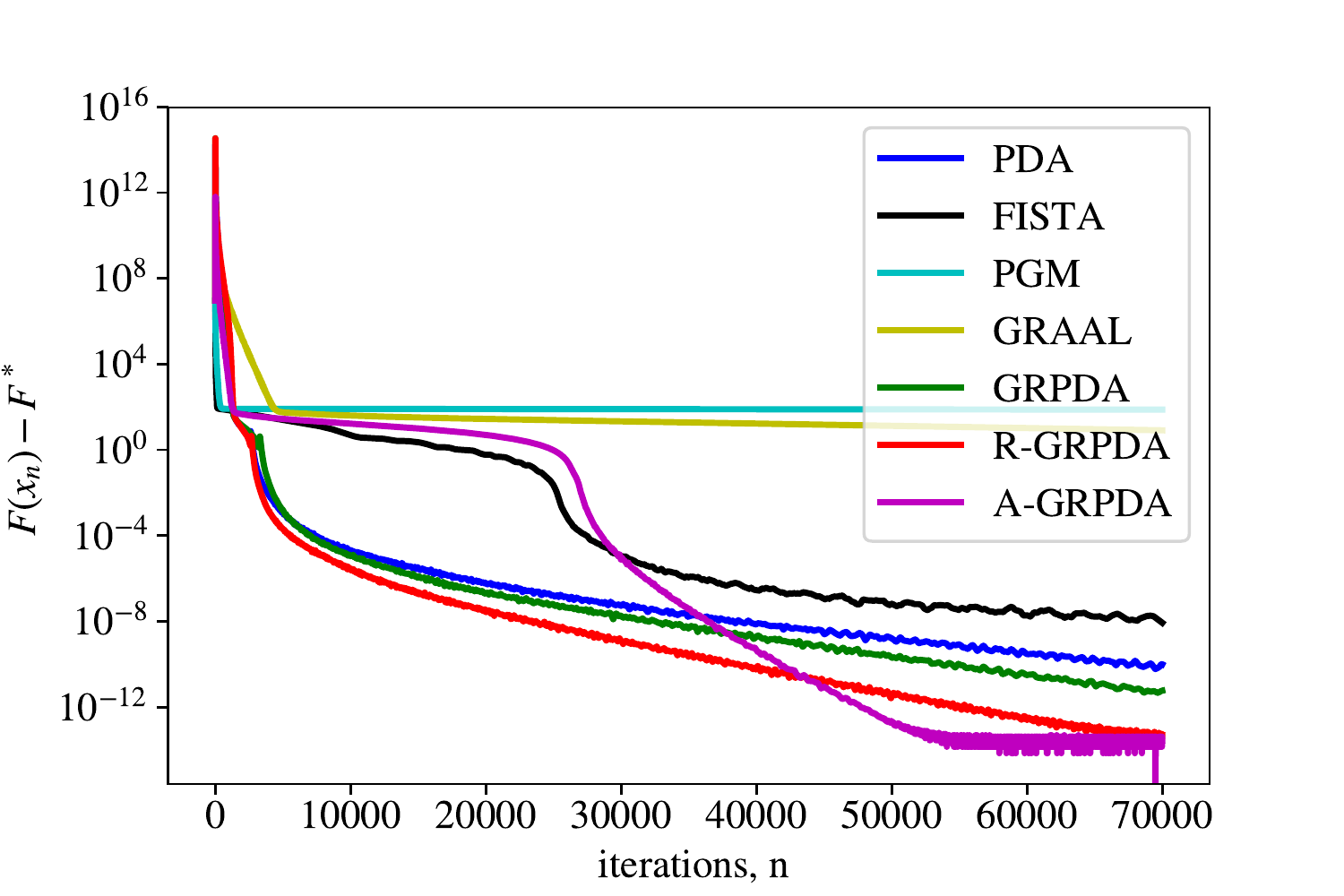}}
\caption{Numerical results for the LASSO problem.} %Comparison of $F(x) - F^*$ for solving Problem \ref{pro_1} with $p=200$, $q=1000$ and $s=10$. }
\label{Fig 1} %% label for entire figure
\end{figure}

%\comm{Since the notation $\phi$ is already used to denote the golden ratio, it might be better not to use it again to denote other things (though no confusion at all here). Need to change the $y$-axes labels by $F$.}
%\reply{It is revised.}
%\comm{The tested case for LASSO seems to be an easy case: $p = 200$, $q=1000$ and $s=10$, and data were randomly generated. This might be a criticism of referee.}\reply{The result showing efficiency of GRPDA with different $\psi$ is added. Some cases with larger sizes are computed.}

\begin{prob}[Nonnegative least-squares problem]\label{pro_2}
The nonnegative least-squares problem aims to find $x\in\R^q_+$ (the nonnegative orthant) such that $\|Kx-b\|$ is minimized for given $K\in\R^{p\times q}$ and $b\in\R^p$, i.e., $\min_{x\in\R^q} F(x):= \frac 1 2 \|Kx-b\|^2 + \iota_{\R^q_+}(x)$.
\end{prob}

Similar to the LASSO problem, the nonnegative least-squares problem can also be solved by A-GRPDA and R-GRPDA since the data fitting term $\frac 1 2 \|Kx-b\|^2$ remains the same. The only difference is that the $\ell_1$-norm $\mu\|\cdot\|_1$ in LASSO is now replaced by the indicator function $\iota_{\R^q_+}(\cdot)$. Since the proximal mapping of the indicator function $\iota_{\R^q_+}(\cdot)$ is just the projection on $\R^q_+$, all the compared algorithms keep feasibility of the nonnegativity constraint.
As such,  $F(x_n) = \frac 1 2 \|Kx_n-b\|^2$ since $x_n\geq 0$ is always satisfied. % by all the compared algorithms.
For this experiment, we consider two types of data described below.
\bi
\item[(i)] Real data from the Matrix Market library\footnote{https://math.nist.gov/MatrixMarket/data/Harwell-Boeing/lsq/lsq.html.}. Two instances were tested, i.e.,  ``illc1033" and ``illc1850", where $K\in\R^{p\times q}$ is sparse and has sizes $(p,q) = (1033,320)$ and $(1850, 712)$, respectively.
    These matrices were used in the testing of the famous LSQR algorithm \cite{PaS82acm} and are much more ill-conditioned than the other two instances ``well1033" and ``well1850" available at the library.
    The entries of $b\in \bR^p$ were drawn independently from $\cN (0, 1)$.

\item[(ii)] Random matrix. $K\in\R^{p\times q}$ has approximately $dpq$ nonzeros, with $d \in(0,1)$ as in \cite{Malitsky2018A},
and $x^*$ is a sparse vector, whose $s$ nonzero entries were drawn uniformly from $[0,100]$. We set $b = Kx^*$, which gives $F^* = 0$.
The following two cases were tested.
\begin{enumerate}
\item[(a)] $p = 1000$, $q = 2000$, $d = 0.5$ and $s = 100$. The nonzero entries of $K$ were generated independently from the uniform distribution in $[0,1]$.
\item[(b)] $p = 10000$, $q = 20000$, $d = 0.01$ and $s = 500$. The nonzero entries of $K$ were generated independently from the normal distribution $\mathcal{N}(0,1)$.
\end{enumerate}

\ei

%\remove{Notice again that $\prox_{\tau f^*}$ is just affine for this problem, we implement GRPDA with $\psi=2$. For PGM and FISTA, we compute $\|K\|$ and set $\alpha=\frac{1}{\|K\|^2}$. For A-GRPDA, we adopt $\beta_0=1$ but different $\gamma$, $\gamma=1$ for ``ILLC1033",  $\gamma=0.01$ for ``ILLC1850" and random data. } \comm{Again, not clear why $\gamma$ is tunable.}

For GRPDA and R-GRPDA, we set $\beta=1$, except for the random data with $p=1000$, for which we set $\beta=25$ as in \cite{Malitsky2018A}. The same as the LASSO case, all the algorithms were initialized at $x_0 = 0$ and $y_0 = -b$.

\begin{figure}[!htp]
\centering
\subfigure[illc1033.]{
\includegraphics[width=0.45\textwidth]{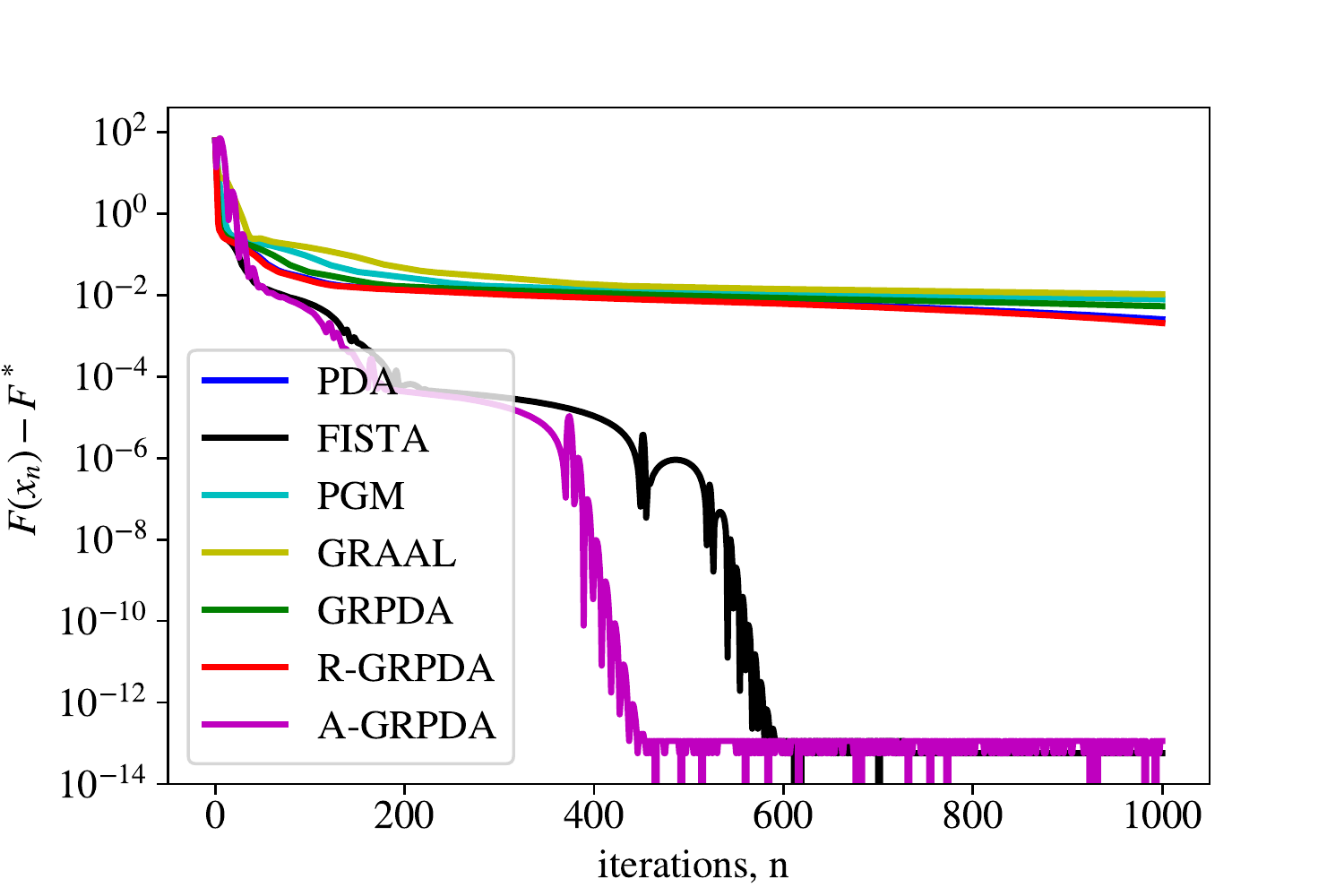}}
\subfigure[illc1850.]{
\includegraphics[width=0.45\textwidth]{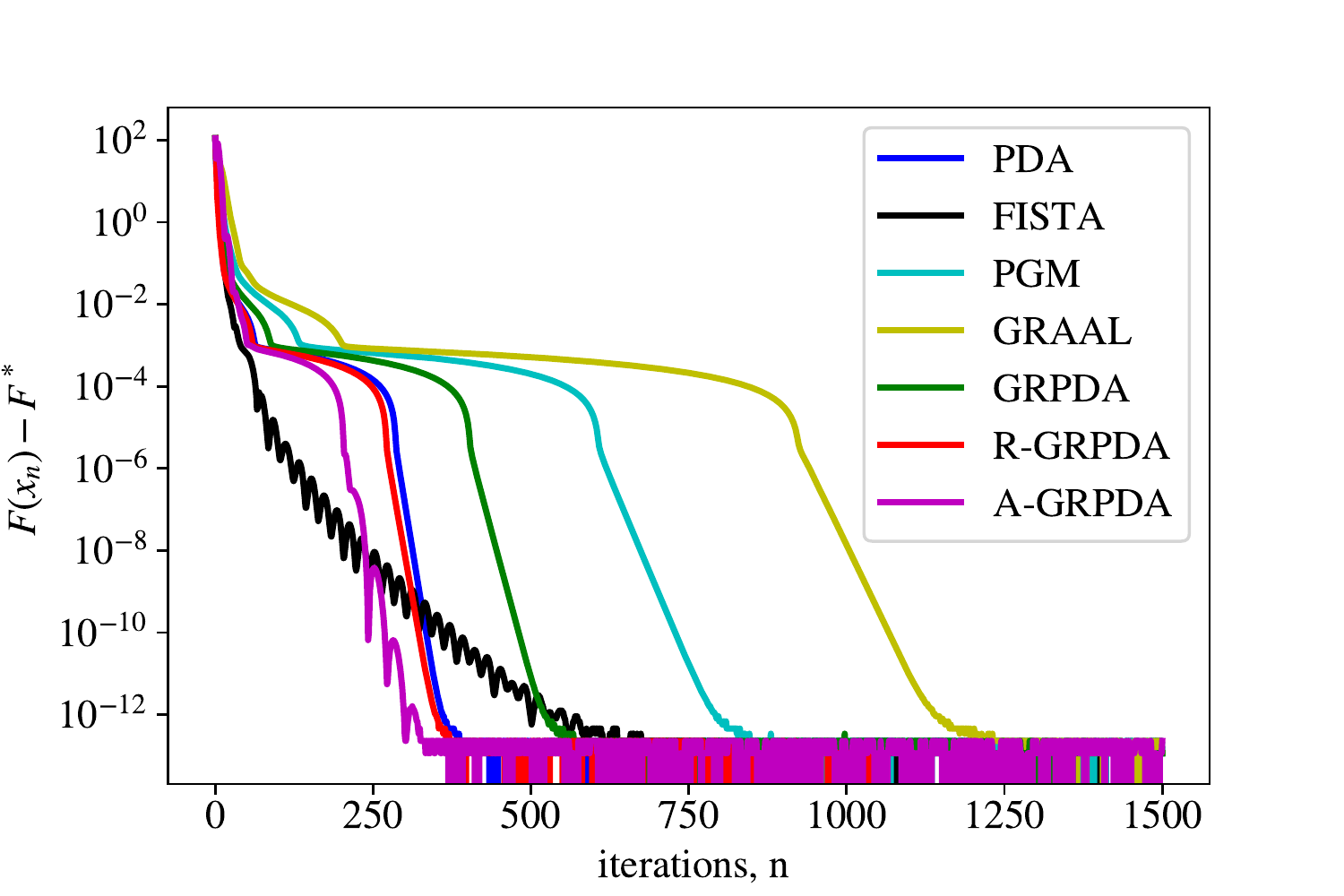}}
\caption{%Comparison of $F(x)-F^*$ for solving Problem \ref{pro_2} with real data.
Numerical results for nonnegative least-squares problem: Real data.}
\label{Fig ILL} %% label for entire figure
\end{figure}
%\comm{It seems that for real data GRPDA and R-GRPDA perform poorly for the let case. Is it possible to accelerate them?}\reply{I am not entirely sure the reason. Maybe the real data is much more ill, accelerated methods performs better.}

\begin{figure}[!htp]
\centering
\subfigure[$p = 1000$]{
\includegraphics[width=0.45\textwidth]{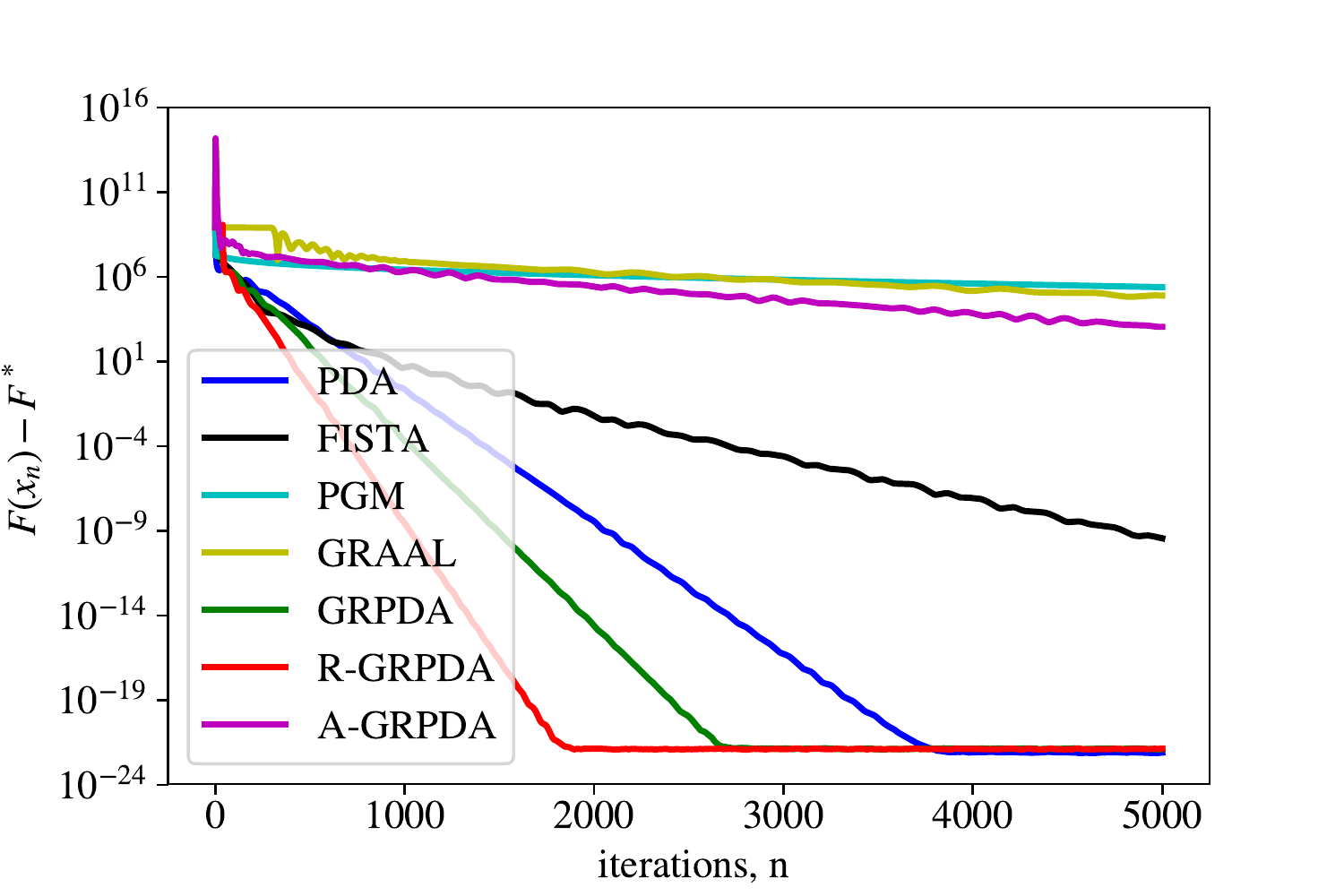}}
\subfigure[$p = 10000$]{
\includegraphics[width=0.45\textwidth]{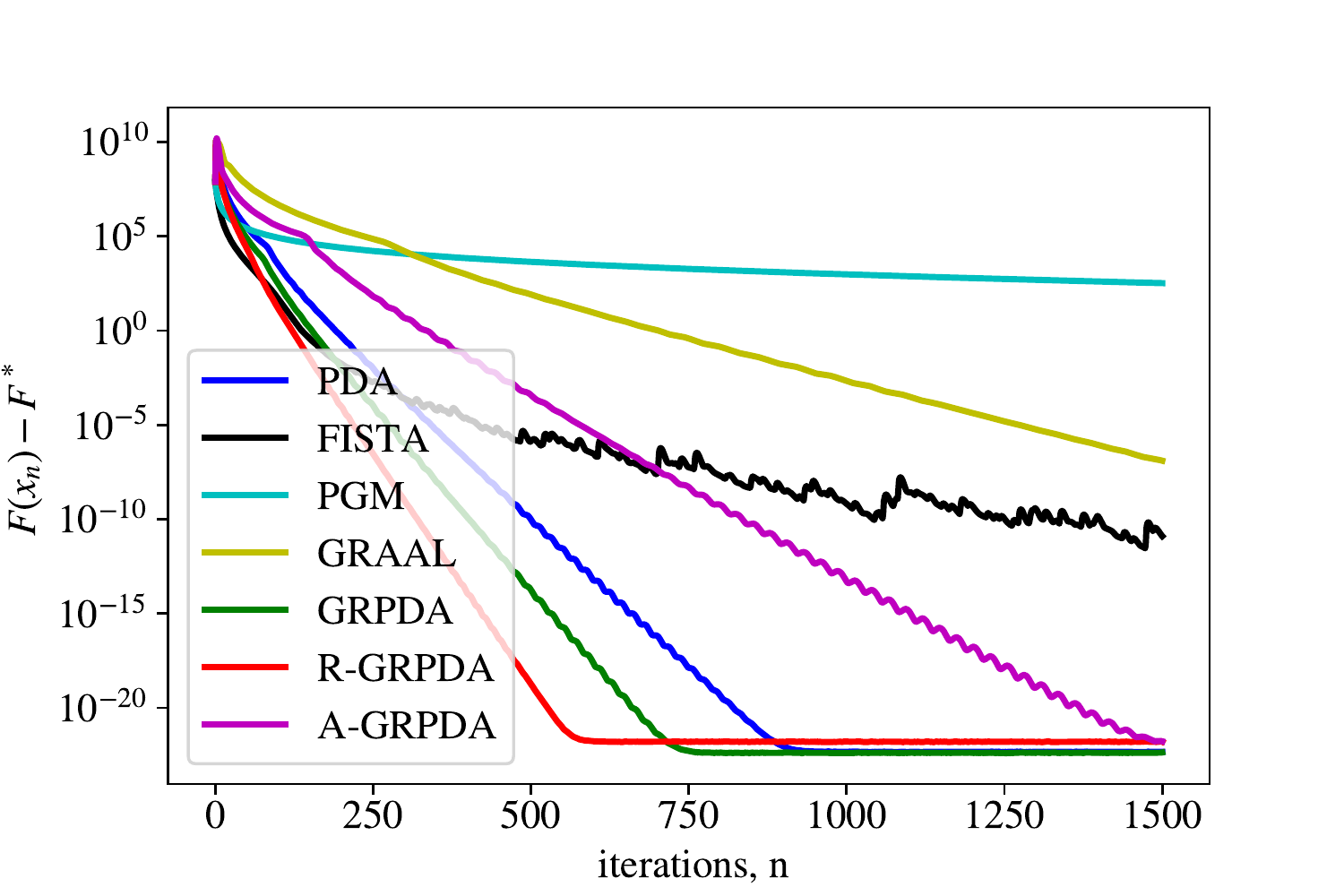}}
\caption{%Comparison of $F(x)-F^*$ for solving Problem \ref{pro_2} with random data.
Numerical results for nonnegative least-squares problem: Random data.}
\label{Fig NLS} %% label for entire figure
\end{figure}

%\comm{Change $m$ to $p$ in $x$ label, $\phi$ to $F$ in $y$ label}

%\comm{For real data and random data, the performance differs very significantly. For real data, A-GRPDA performs best, followed by FISTA, others are less efficient. While for random data, A-GRPDA get worse a lot and R-GRPDA performs best. As such, it seems difficult to evaluate the algorithm and draw a conclusion.}\reply{For the real data, I select much more ill-condition cases to show that the accelerated version can be efficient.}

Similar to the LASSO case, the decreasing behavior of $F(x_n)-F^*$ as the algorithms proceeded is presented in Figure \ref{Fig ILL} for real data and Figure \ref{Fig NLS} for random data. It can be observed from Figure \ref{Fig ILL} that A-GRPDA performs the best, followed by FISTA. In comparison, for random data R-GRPDA performs the best, followed by GRPDA and PDA, as shown in Figure \ref{Fig NLS}. Other algorithms seem to be less efficient. Nonetheless, all the compared algorithms perform favorably and have attained fairly high precision in a reasonable number of iterations.

\begin{prob}[Minimax matrix game]\label{pro_3}
Let $\D_q = \{x\in\R^q: x\geq 0, \; \sum_{i=1}^q x_i = 1\}$ be the standard unit simplex in $\R^q$, and  $K\in \R^{p\times q}$.
The minimax matrix game problem is given by
\be\label{mm_pro}
\min_{x \in \R^q}\max_{y\in \R^p} \iota_{\D_q}(x) + \lr{Kx, y} - \iota_{\D_p}(y).
\ee
\end{prob}

Apparently, \eqref{mm_pro} is a special case of \eqref{pd-prob} with $g = \iota_{\D_q}$ and $f^* = \iota_{\D_p}$.
We note that the projection onto $\D_q$ can be computed efficiently, see, e.g., \cite[Corollary 6.29]{Beck2017book}.
Therefore, any algorithms depending on the proximal operators of  $\iota_{\D_q}$ can be implemented efficiently.
In our implementation, we used the algorithm from \cite{Duchi2011Diagonal} to compute the projection onto the standard unit simplex.
To compare different algorithms, we used the primal-dual gap  function as defined in \eqref{G}, which can be easily computed for a feasible pair $(x, y) \in\D_q \times \D_p$ by $G(x, y) := \max_i(Kx)_i - \min_j(K^\top y)_j$. Here subscript $i$ (or $j$) denotes the $i$th (or $j$th) component of the underlying vector.

For this minimax matrix game problem, the only relevant algorithms discussed at the beginning of this section are
PDA, GRPDA and GRAAL. The rest are not applicable. In this experiment, we set $\beta = 1$ for PDA and GRPDA.
The initial point for all algorithms was set to be $x_0 =\frac{1}{q}(1,\ldots,1)^\top$ and $y_0 =\frac{1}{p}(1,\ldots,1)^\top$.
We tested two types of random matrix $K\in \bR^{p\times q}$ with $seed=50$,  i.e.,
(i) $(p,q) = (100,100)$ and all entries of $K$ were generated independently from the uniform distribution in $[-1, 1]$, and (ii)
$(p,q) = (100,500)$ and all entries of $K$ were generated independently from the normal distribution $\cN (0,1)$.

\begin{figure}[htp]
\centering
\subfigure[Case (i).]{
\includegraphics[width=0.45\textwidth]{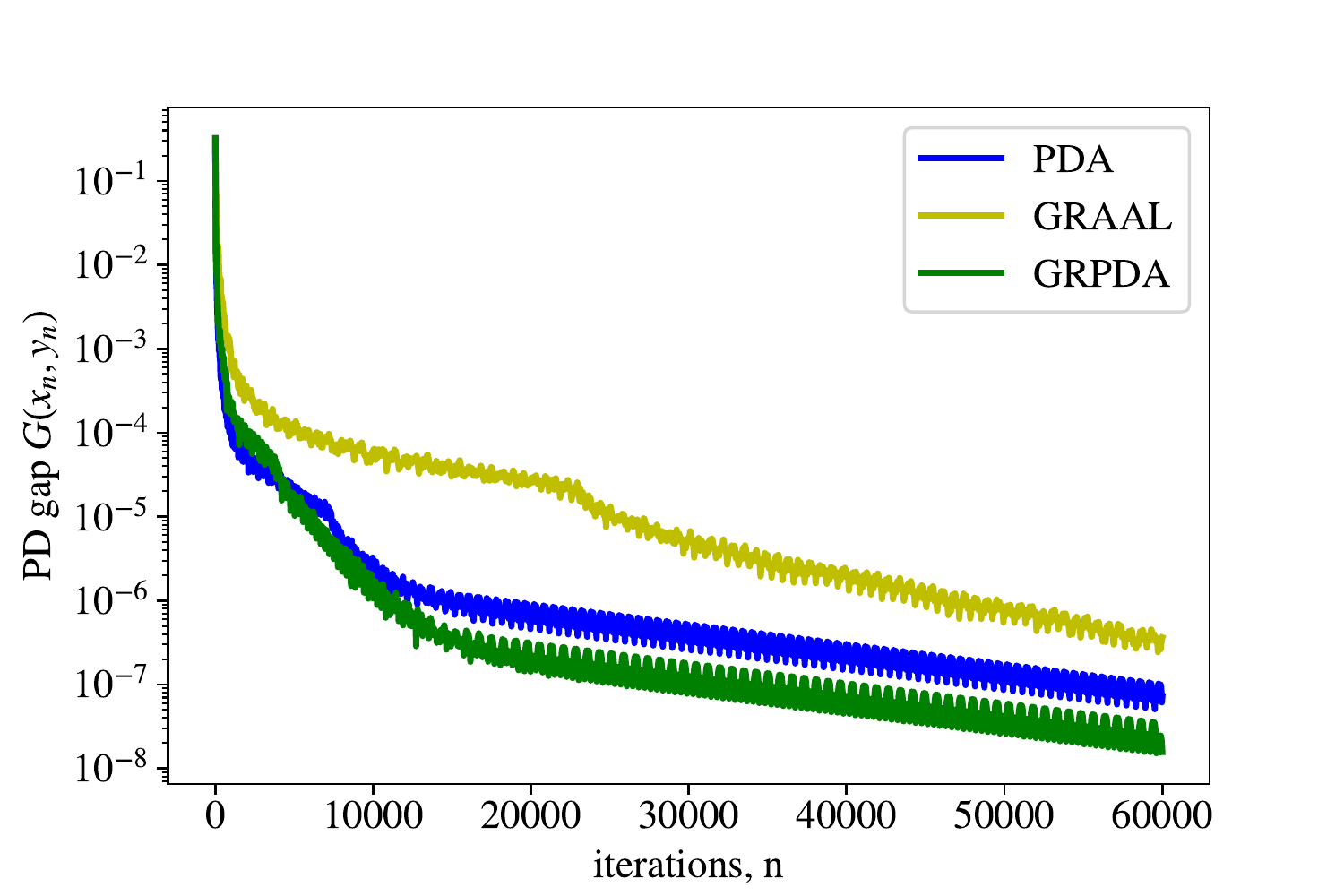}}
\subfigure[Case (ii).]{
\includegraphics[width=0.45\textwidth]{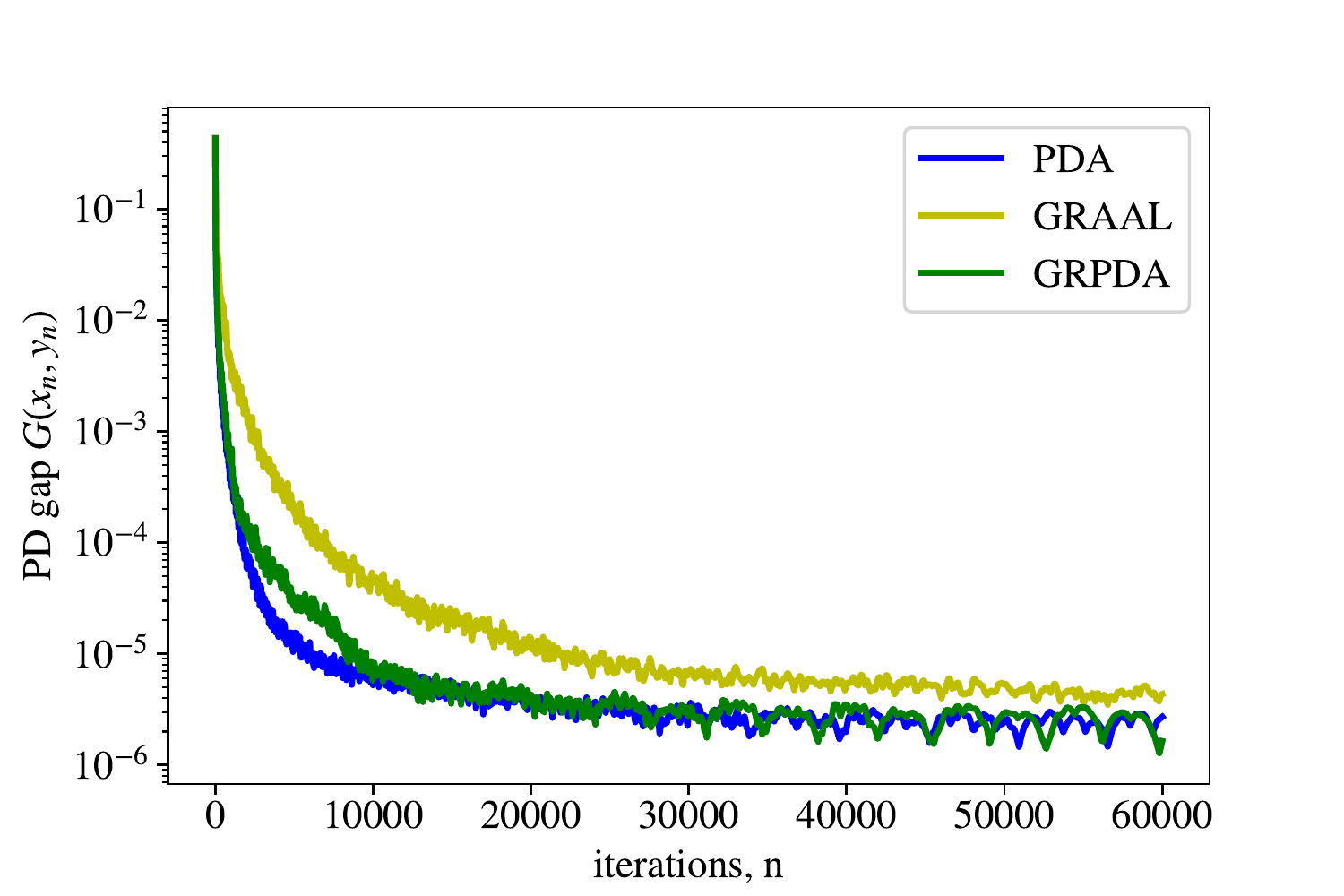}}
\caption{%Comparison of PD gap for solving Problem \ref{pro_3}.
Numerical results for minimax matrix game problem.
}
\label{Fig 5} %% label for entire figure
\end{figure}

The decreasing behavior of the primal-dual gap function as the algorithms proceeded is given in Figure \ref{Fig 5}, from which it can be seen that in both tests GRPDA performs comparably with PDA and both GRPDA and PDA outperform GRAAL. Recall that in this experiment $\beta$ was set to be $1$ and thus the primal step size $\tau$ and the dual step size $\eta$ are equal.
We attribute the faster convergence of GRPDA and PDA, as compared to GRAAL, to the Guass-Seidel nature of GRPDA \eqref{GRPDA} and PDA \eqref{pda_basic}, as compared to the Jacobian type iteration GRAAL \eqref{gr-alg2}.

%\comm{Maybe change the $y$ axis labels to $G(x,y)$ or ``PD gap $G(x,y)$"?}\reply{It is revised.}

\section{Conclusions}
\label{sec_conclusion}
In this paper, motivated by the recent work \cite{Malitsky2019Golden}, we proposed, analyzed and tested a golden ratio primal-dual algorithm (GRPDA), which can be viewed as a new remedy to the classical Arrow-Hurwicz method. Different from the PDA in \cite{Chambolle2011A}, which uses an extrapolation or inertial step, the proposed GRPDA uses a convex combination technique originally introduced in \cite{Malitsky2019Golden}.
Global iterate convergence and $\cO(1/N)$ ergodic rate of convergence, measured by the primal dual gap function, are established.
When either $g$ or $f^*$ is strongly convex, we managed to modify the algorithm so that it enjoys a faster $\cO(1/N^2 )$ ergodic convergence rate.
For two widely used special cases, i.e., regularized least-squares problem and linear equality constrained problem,
we further show  via spectral analysis that the algorithmic parameter $\psi$ can be enlarged from $(1,\phi]$ to $(1,2]$, which allows larger primal and dual step sizes.
Moreover, in the fixed point perspective the iterative mapping is $2/3$-averaged and thus a relaxation step can be taken with parameter $\rho\in (0,2/3)$.
Our preliminary numerical results on LASSO, nonnegative least-squares and minimax matrix game problems show that the proposed algorithms perform favorably. In particular, GRPDA with fixed step sizes is comparable with PDA in general, and the relaxed and accelerated variants can achieve superior performance.

Some interesting issues remain to be investigated. For example, how to modify the proposed algorithms so that they are suitable to the cases when the operator norm $\|K\|$ is hard to evaluate, how to adapt the proposed algorithms to more general settings, including the general linearly constrained separable convex optimization problem $\min_{x,y} \{f(x) + g(y): Ax + By = b\}$, or when the operator $K$ is nonlinear. Moreover, for large scale finite sum problem, which are abundant in machine learning, it is interesting to design and analyze their stochastic and/or incremental counterparts. We leave these issues for further investigations.

%(a) how to extend GRPDA to \eqref{pd-prob} with non-linear operator $K$; (b) how to derive stochastic extension for GRPDA, when solving large problems that are encountered in many real world applications; (c) it is observed that using some inertia and relaxation for the optimization algorithm often accelerates convergence. In our golden ratio PDA, it will be in particular interesting to do so.

%\begin{acknowledgements}
%The research of Xiaokai Chang was supported by  the Hongliu Foundation of First-class Disciplines of Lanzhou University of Technology.
%\end{acknowledgements}

%\comm{How to control the reference order? The present ascending order is not preferred.}
%\comm{
%[27]: too much information.
%[21]: incomplete information;
%[20]: maybe updated.
%[15]: ``C" between page numbers.
%
%}

%\bibliographystyle{unsrt,siam,abbrv}
\bibliographystyle{abbrv}
\bibliography{cxktex}

\end{document}